\documentclass[a4paper,12pt]{article}

\usepackage[table,dvipsnames,svgnames,x11names,hyperref]{xcolor}
\usepackage[utf8]{inputenc}

\usepackage{lmodern}
\usepackage[T1]{fontenc}

\usepackage[top=1in, bottom=1.25in, left=1.25in, right=1.25in]{geometry}
\usepackage{amsmath,amssymb,amsthm}

\usepackage{tikz}
\usepackage{array}
\usepackage{enumerate}
\usepackage{graphicx}
\usepackage{float}
\usepackage{caption}
\usepackage{subcaption}
\usepackage[colorlinks=true, allcolors=MidnightBlue]{hyperref}
\usepackage{rotating}
\usepackage{listings}
\usepackage{tikz-cd}
\usepackage{tikzpeople}

\newtheorem{thm}{Theorem}

\newtheorem{lem}[thm]{Lemma}
\newtheorem{prop}[thm]{Proposition}

\theoremstyle{definition}

\newtheorem*{construction}{Construction}

\theoremstyle{remark}
\newtheorem*{rem}{Remark}

\newtheorem*{question}{Question}

\author{Sel\c{c}uk Kayacan
  \thanks{This study was supported by Scientific and Technological Research Council of Turkey (TUBITAK) under the Grant Number 122F490. The author thanks to TUBITAK for their supports.}
}

\title{Subrack lattices of finite solvable and metacyclic groups}
\date{}

\begin{document}

\maketitle

\small

\begin{center}
  Bahçeşehir University, Faculty of Engineering\\ and Natural Sciences,
  Istanbul, Turkey\\
  {\it e-mail:} \href{mailto:selcuk.kayacan@bau.edu.tr}{selcuk.kayacan@bau.edu.tr}
\end{center}

\begin{abstract}
  
  A group $G$ with conjugation operation is a rack. We call such racks \emph{group racks}. In this paper we study finite group racks via their subrack lattices. Heckenberger, Shareshian, and Welker proved that the isomorphism type of the subrack lattice of a finite group determines whether the group is solvable. Our first result shows that if $G$ is a finite solvable group and $H$ is a finite group whose subrack lattice is isomorphic to that of $G$, then $H$ is solvable and the derived length of $H$ has the same derived length as $G$. Our second result is that if $G$ is a finite metacyclic group and $H$ is a group whose subrack lattice is isomorphic to that of $G$, then $H/Z(H)$ is metacyclic. As a further application of our analysis of finite metacyclic groups, we answer a question of Heckenberger, Shareshian, and Welker in the affirmative by constructing two finite groups with isomorphic subrack lattices that are not isomorphic as racks.
  


\end{abstract}

\section{Introduction}

A \emph{rack} $X$ is a set with a binary operation $\triangleright\colon X\times X\to X$ satisfying the following two conditions:
\begin{itemize}
\item[(A1)] for all $a,b,c\in X$ we have $a\triangleright (b\triangleright c) = (a\triangleright b)\triangleright (a\triangleright c)$, and
\item[(A2)] for all $a,c\in X$ there exists a unique $b\in X$ such that $a\triangleright b = c$.
\end{itemize}
If, further, the following axiom holds we say $X$ is a \emph{quandle}:
\begin{itemize}
\item[(A3)] for all $a\in X$ we have $a\triangleright a = a$.
\end{itemize}

Let $G$ be a finite group and $S$ be a subset of $G$ that is closed under the conjugation operation $a\triangleright b := aba^{-1}$. The set $S$ with the binary operation $\triangleright$ forms a quandle. We call these racks \emph{conjugation racks}. The following three classes of conjugation racks deserve special attention:
\begin{itemize}
\item the \emph{group rack} $(G,\triangleright)$,
\item the \emph{conjugacy class rack} $(C,\triangleright)$, where $C$ is a conjugacy class of $G$, and
\item the \emph{$p$-power rack $(G_p,\triangleright)$}, where $p$ is a prime dividing the order of $G$ and $G_p$ is the set of all elements of $G$ whose order is a power of $p$.
\end{itemize}

The poset of subracks of a finite rack $X$ ordered by inclusion forms a lattice with unique maximum element $X$ and unique minimum element $\emptyset$ (see \cite[Lemma~2.1]{HSW19}). We call this lattice the \emph{subrack lattice} of $X$ and denote it by $\mathcal{R}(X)$.

The order complex of the subrack lattice of conjugation racks has interesting geometrical properties. For example, basic group theoretical arguments show that, for a group rack $G$ the maximal elements of $\mathcal{R}(G)$ are the union of all but one conjugacy classes of $G$ and this fact can be used to prove that the order complex of $\mathcal{R}(G)\setminus \{G,\emptyset\}$ is homotopy equivalent to a $(c-2)$-sphere, where $c$ is the number of conjugacy classes of $G$ (see \cite[Proposition~1.3]{HSW19}). Further homotopy properties of the subrack lattices of conjugacy class racks and $p$-power racks are studied in \cite{Kay26a,Kay26b}.

Heckenberger, Shareshian, and Welker showed in \cite{HSW19} that the subrack lattice $\mathcal{R}(G)$ of a group rack $G$ contains essential information about the group itself.

\begin{thm}[see {\cite[Theorem~1.1]{HSW19}}]\label{thm:1.1}
  Let $G$, $H$ be ﬁnite groups satisfying $\mathcal{R}(G)\cong \mathcal{R}(H)$.
  \begin{enumerate}
  \item If $G$ is abelian, then $H$ is abelian.
  \item If $G$ is nilpotent, then $H$ is nilpotent.
  \item If $G$ is supersolvable, then $H$ is supersolvable.
  \item If $G$ is solvable, then $H$ is solvable.
  \item If $G$ is simple, then $H$ is simple.
  \end{enumerate}
\end{thm}

Moreover, the group $H$ is nilpotent of class $c$ if $G$ is nilpotent of class $c$ (see \cite[Theorem~3.3]{Kay21}) and $H$ is $p$-nilpotent if $G$ is $p$-nilpotent and $\mathcal{R}(G)$ satisfies a certain condition (see \cite[Theorem~4.5]{Kay21}).

The \emph{derived length} of a solvable group $G$ is the smallest integer $n$ such that the $n$th higher commutator subgroup $G^{(n)}$ is trivial. Since solvability is determined by the subrack lattice, it is natural to ask whether finer invariants of solvable groups can also be recovered. In this paper we show that the methods developed in \cite{Kay21} also recover this invariant from the subrack lattice of the corresponding group rack.

\begin{thm}\label{thm:main_derived}
  If $G$ is a finite solvable group with derived length $\ell$, and if $H$ is a group such that $\mathcal{R}(G)\cong \mathcal{R}(H)$, then $H$ is solvable with derived length $\ell$.
\end{thm}

Finite solvable groups admit subnormal series whose factor groups are cyclic. Among finite solvable groups, metacyclic groups form one of the first genuinely nontrivial classes. Recall that a group $G$ is called \emph{metacyclic} if it has a cyclic normal subgroup $N$ such that the quotient $G/N$ is also cyclic. Our second main result is the following.

\begin{thm}\label{thm:main_metacyclic}
  Let $G$ and $H$ be finite groups. If $G$ is metacyclic and $\mathcal{R}(G)\cong \mathcal{R}(H)$, then $H/Z(H)$ is metacyclic.
\end{thm}

One cannot, however, expect metacyclicity itself to be detected by $\mathcal{R}(G)$. For example, the abelian groups 
  $$ G = C_4\times C_2\; \text{ and }\; H = C_2\times C_2\times C_2.$$
  have isomorphic subrack lattices, while the former is metacyclic and the latter is not. Thus the most one can reasonably hope for is the recovery of metacyclic structure modulo central factors. A natural next question is whether isomorphic subrack lattices force the corresponding inner automorphism groups to be isomorphic. For instance, the groups $D_8$ and $Q_8$ have isomorphic subrack lattices and satisfy
  $$ D_8/Z(D_8)\cong C_2\times C_2\cong Q_8/Z(Q_8). $$
In \cite{HSW19} Heckenberger, Shareshian, and Welker raised the following stronger variant of this question.

\begin{question}[see {\cite[Question~5.2]{HSW19}}]\label{que:5.2}
  Are there two groups $G$, $H$ which have isomorphic subrack lattices but are non-isomorphic as racks?
\end{question}

We write $G\cong_{\mathsf{rk}}H$ if the groups $G$ and $H$ are isomorphic as racks. Our study of finite metacyclic groups leads to an affirmative answer to this question.

\begin{prop}\label{prop:ex}
  Let
  \begin{align*}
    G_0 &= \langle a,b\mid a^{91} = 1,\ b^3 = 1,\ bab^{-1} = a^9\rangle, \\
    H_0 &= \langle \alpha,\beta\mid \alpha^{91} = 1,\ \beta^3 = 1,\ \beta\alpha\beta^{-1} = \alpha^{16}\rangle.
  \end{align*}
  Then $\mathcal{R}(G_0)\cong \mathcal{R}(H_0)$ but $G_0\not\cong_{\mathsf{rk}}H_0$.
\end{prop}

The paper is organized as follows. In Section~\ref{sec:pre} we review the basic facts on racks needed in the sequel. In Sections~\ref{sec:sol} and~\ref{sec:meta} we prove Theorem~\ref{thm:main_derived} and Theorem~\ref{thm:main_metacyclic}, respectively. In Section~\ref{sec:iso} we construct a bijection from $G_0$ to $H_0$ that induces a lattice isomorphism between their subrack lattices and prove Proposition~\ref{prop:ex}. In the Appendix we present GAP code \cite{GAP4} that verifies the lattice isomorphism from $\mathcal{R}(G_0)$ to $\mathcal{R}(H_0)$.

\section{Preliminary Facts}\label{sec:pre}

Let $X$ be a rack and, for each element $x\in X$, let $\phi_x\colon X\to X$ be the map that takes $y$ to $x\triangleright y$. By axiom (A2), the map $\phi_x$ is a permutation of $X$; hence, by axiom (A1), an automorphism of $X$. Let $\Phi(X) := \{\phi_x\colon x\in X\}$. The inner automorphism group $\mathrm{Inn}(X)$ is the normal subgroup of the automorphism group $\mathrm{Aut}(X)$ that is generated by $\Phi(X)$. We say $X$ is \emph{faithful} if the map $\Phi\colon X\to \Phi(X)$ is injective. For basic facts on racks, see \cite{AG03}; in particular, Lemma~1.7, Lemma~1.9, and Proposition~3.2 there provide alternative statements and proofs of the following proposition. The formulation below appears in \cite[Proposition~2]{Kay25}.

\begin{prop}\label{prop:act}
  Let $X$ be a rack. Then the following statements hold.
  \begin{enumerate}[(i)]
  \item The map $\Phi\colon X\to \Phi(X),\; x\mapsto \phi_x$ is a rack morphism which is an isomorphism if $X$ is faithful.
  \item If $X$ is faithful, then the action of $\mathrm{Inn}(X)$ on $X$ by automorphisms is isomorphic to its action on $\Phi(X)$ by conjugation, and the center of $\mathrm{Inn}(X)$ is trivial.
  \item If $X$ is a conjugation rack and $G := \langle X \rangle$, then the map $\Phi\colon X\to \Phi(X)$ extends to a group homomorphism
  $$\Phi\colon G\to \mathrm{Inn}(X),$$
  and $\mathrm{Inn}(X)\cong G/Z(G)$.
  \end{enumerate}
\end{prop}

In this paper we are mainly interested in finite group racks. If $X=G$ is a group endowed with its conjugation rack structure, then the inner automorphism group of the rack $X$ coincides with the usual inner automorphism group of $G$.

\begin{lem}[see {\cite[Lemma~2.8]{HSW19}}]\label{lem:2.8}
  Let $G$ be a finite group and let $X$ be a subrack of $G$ that contains all non-central conjugacy classes of $G$. A subrack $M$ of $X$ is maximal in $\mathcal{R}(X)$ if and only if $M$ is the union of all but one of the conjugacy classes of $G$ contained in $X$.
\end{lem}

As a consequence of the above Lemma, the combinatorial structure of the subrack lattice $\mathcal{R}(G)$ determines the conjugacy classes of $G$. Another structural fact about finite conjugation racks is that the atoms of the corresponding subrack lattices are exactly the $1$-element subsets of the rack. This, in particular, implies that if $\psi\colon \mathcal{R}(G)\to \mathcal{R}(H)$ is a lattice isomorphism, then the orders of the groups $G$ and $H$ must be the same. In the sequel we will use both of these facts freely without further reference.

\section{Solvable groups}\label{sec:sol}

Let $G$ be a finite group, and let $\mathsf{A}(G)$ denote the set of all elements in the subrack lattice $\mathcal{R}(G)$ satisfying the following property: An element $A\in \mathcal{R}(G)$ belongs to $\mathsf{A}(G)$ if and only if the interval $[\emptyset,A]$ is a maximal Boolean algebra in $\mathcal{R}(G)$ (that is, if $[\emptyset,B]$ is a Boolean algebra for some $B\in \mathcal{R}(G)$ containing $A$ then $A=B$). One can easily observe that an element $A$ of the subrack lattice $\mathcal{R}(G)$ is a maximal abelian subgroup of $G$ if and only if $A$ belongs to $\mathsf{A}(G)$. Since the combinatorial structure of $\mathcal{R}(G)$ determines the conjugacy classes of $G$, we can also characterize the maximal normal abelian subgroups of $G$ via the combinatorial properties of $\mathcal{R}(G)$.

\begin{lem}[see {\cite[Lemma~3.1]{Kay21}}]\label{lem:normal}
  For a finite group $G$, an element $N\in\mathcal{R}(G)$ is a maximal normal abelian subgroup of $G$ if and only if $N$ is the union of all conjugacy classes of $G$ that are contained in some fixed $A\in \mathsf{A}(G)$.
\end{lem}

Given a subset $S$ of $G$, we denote by $\langle S\rangle_{\mathsf{rk}}$ the subrack of the group rack $G$ generated by the elements in $S$.

\begin{lem}[see {\cite[Lemma~3.2]{Kay21}}]\label{lem:coset}
  Let $G$ be a finite group and $N$ be a normal subgroup of $G$. Then any coset of $N$ is a subrack of $G$. Moreover, for any subset  $\{a_1N,\dots,a_kN\}$ of $G/N$, the join of those cosets in $\mathcal{R}(G)$ is the union of the cosets of $N$ containing $\langle a_1,\dots,a_k \rangle_{\mathsf{rk}}$.
\end{lem}

Let $N$ be a normal subgroup of $G$. By Lemma~\ref{lem:coset} the subrack lattice $\mathcal{R}(G/N)$ is isomorphic to a subposet of $\mathcal{R}(G)$. Let $H$ be a group whose subrack lattice is isomorphic to the subrack lattice of $G$. The following result will be used to construct an isomorphic copy of the poset $\mathcal{R}(G/N)\setminus \{\emptyset\}$ in the subrack lattice $\mathcal{R}(H)$.

\begin{lem}\label{lem:partition}
Let $\mathcal{P} = \{P_1,\dots,P_m\}$ and $\mathcal{Q} = \{Q_1,\dots,Q_m\}$ be some partitions of a set $S$ of size $mn$ into parts each having the same size $n$. Then there is a bijection $f\colon \mathcal{P}\to \mathcal{Q}$ such that $$ f(P_i)\cap P_i \neq \emptyset\; \text{ for each }\; 1\leq i\leq m. $$
\end{lem}

\begin{proof}
  The proof is by induction on the number of parts. Obviously, the assertion holds when the partition consists of one element, i.e. when the size of $S$ is $n$. Suppose the assertion holds for sets of size $(m-1)n$. We want to show that the assertion holds when the size of $S$ is $mn$. There are two cases that we shall consider.

  \medskip
  \noindent
  \emph{Case I:} $P_{i_0}=Q_{j_0}$ for some $1\leq i_0,j_0\leq m$. Observe that the partitions $\mathcal{P}_0:= \mathcal{P}\setminus \{P_{i_0}\}$ and $\mathcal{Q}_0:= \mathcal{Q}\setminus \{Q_{j_0}\}$ are the partitions of a set of size $(m-1)n$ and by the inductive hypothesis there exists a function $g\colon \mathcal{P}_0\to \mathcal{Q}_0$ having the desired property. Hence, the function
  \begin{align*}
    &f(P_i) =
      \begin{cases}
        g(P_i) & \quad \text{if }\; i\neq i_0 \\
        Q_{j_0} & \quad \text{if }\; i= i_0
      \end{cases}
  \end{align*}
  satisfies the property $f(P_i)\cap P_i \neq \emptyset$ for all $1\leq i\leq m$.

  \medskip
  \noindent
  \emph{Case II:} $P_i\neq Q_j$ for all $1\leq i,j\leq m$. Fix two index values $i_0$ and $j_0$ such that $P_{i_0}\cap Q_{j_0}\neq \emptyset$. We construct a partition $\mathcal{Q}'=\{Q'_1,\dots,Q'_m\}$ in the following way: Initially $Q'_j=Q_j$ for all $1\leq j\leq m$. If $P_{i_0}\cap Q_j \neq \emptyset$ and $j\neq j_0$, we swap elements of $P_{i_0}\cap Q_j$ with some elements of $Q_{j_0}\setminus P_{i_0}$ in an arbitrary fashion to form $Q'_j$ and $Q'_{j_0}$. Iterating this process at the final step $P_{i_0}$ and $Q'_{j_0}$ would be the same sets. Then, by Case I, there exists a function $f'\colon \mathcal{P}\to \mathcal{Q}'$ satisfying the property $f'(P_i)\cap P_i \neq \emptyset$ for all $1\leq i\leq m$. Observe that the induced function $f\colon \mathcal{P}\to \mathcal{Q}$ would also have the same property as $f(P_{i_0})=Q_{j_0}$ and $P_i\cap Q_j\neq \emptyset$ whenever $P_i\cap Q'_j\neq \emptyset$ and $i\neq i_0$.
\end{proof}

\begin{rem}
  Consider the bipartite graph constructed in the following way: The vertex set is $\mathcal{P}\sqcup \mathcal{Q}$ and there is an edge between two vertices if the intersection of the corresponding sets is not empty. Obviously, the number of neighbors of a subset $\mathcal{S}$ of $\mathcal{P}$ is at least the cardinality of $\mathcal{S}$. Thus, by Hall's marriage theorem (see \cite[Theorem 2.1.2]{Die05}), Lemma~\ref{lem:partition} follows.
\end{rem}

A partition $\mathcal{C} = \{C_1,\dots,C_m\}$ of a finite group $G$ will be called an \emph{hypothetical coset partition} if it satisfies the following three conditions:
\begin{itemize}
\item[(C1)] Each element of $\mathcal{C}$ has the same number of elements.
\item[(C2)] The element $C_1\in\mathcal{C}$ is a union of some conjugacy classes of $G$ and $C_1$ contains a central element.
\item[(C3)] For a subset $I$ of $[m]$ the following implication holds: For every choice of representatives $c_i\in C_i$ for $i\in I$, the join of the elements $C_i$, $i\in I$, in $\mathcal{R}(G)$ is the union of elements of $\mathcal{C}$ containing $\langle \{c_i\colon i\in I\}\rangle_{\mathsf{rk}}$.
\end{itemize}
Notice that the combinatorial structure of the subrack lattice $\mathcal{R}(G)$ determines whether a given partition of $G$ is a hypothetical coset partition. If $\mathcal{C}$ is a hypothetical coset partition of $G$, we denote by $\mathsf{J}(\mathcal{C})$ the subposet of $\mathcal{R}(G)$ whose elements are the joins of subsets of elements $\mathcal{C}$.

\begin{proof}[Proof of Theorem~\ref{thm:main_derived}]
  Take an element $A\in \mathsf{A}(G)$ and let $N$ be the union of all conjugacy classes of $G$ that are contained in $A$. By Lemma~\ref{lem:normal} the subrack $N$ is a maximal normal abelian subgroup of $G$ and any maximal normal abelian subgroup of $G$ can be described in this way via the combinatorial structure of $\mathcal{R}(G)$. 

  \medskip
  \noindent
  \emph{Claim. If $N\trianglelefteq G$ is a maximal normal abelian subgroup, then $\mathsf{J}(\mathcal{P})$ is isomorphic to $\mathcal{R}(G/N)\setminus \{\emptyset\}$ for any hypothetical coset partition $\mathcal{P}$ of $G$ with first part $P_1 = N$.}

  \medskip
  By Lemma~\ref{lem:coset} there is a subposet of $\mathcal{R}(G)$ that is isomorphic to $\mathcal{R}(G/N)\setminus \{\emptyset\}$. Clearly, the atoms of this poset determine a partition of $G$ satisfying the conditions (C1), (C2), and (C3). Let $\mathcal{P}$ be a hypothetical coset partition of $G$ with $P_1 = N$. By Lemma~\ref{lem:partition}, there exists a set of representative elements for the partition $\mathcal{P}$ such that those elements are also representatives for the cosets of $N$. Since $\mathcal{P}$ is a hypothetical coset partition, condition (C3) shows that the join operation on the parts of $\mathcal{P}$ agrees with the join operation on the corresponding cosets of $N$, independently of the choice of representatives. Therefore the claim holds.

  \medskip
  Take a hypothetical coset partition $\mathcal{P}$ of $G$ with $P_1 = N$. By the claim above $\mathsf{J}(\mathcal{P}) \cong \mathcal{R}(G/N)\setminus \{\emptyset\}$. Let $\psi\colon \mathcal{R}(G)\to \mathcal{R}(H)$ be a lattice isomorphism and let $\mathcal{Q}$ be the collection of images of the elements in $\mathcal{P}$ under $\psi$. Since $\psi$ is a lattice isomorphism, $\mathcal{Q}$ is a partition of $H$ and satisfies the conditions (C1), (C2), and (C3), hence is a hypothetical coset partition. Furthermore, the element $Q_1 = \psi(N)$ of $\mathcal{Q}$ is a maximal normal abelian subgroup of $H$ as $\psi(N)$ is the union of all conjugacy classes of $H$ that are contained in some element in $\mathsf{A}(H)$ (see Lemma~\ref{lem:normal}). Therefore, by applying the above claim one more time we observe that the subposet $\mathsf{J}(\mathcal{Q})$ of $\mathcal{R}(H)$ is isomorphic to the poset $\mathcal{R}(H/\psi(N))\setminus \{\emptyset\}$. 
  
  Repeating the above process, we obtain a sequence of subposets of $\mathcal{R}(G)$ corresponding to a subnormal series of $G$ with maximal abelian quotients. The isomorphic copies of these subposets in $\mathcal{R}(H)$ yield a corresponding subnormal series for $H$ with maximal abelian factor groups. Since the same argument applies at every step, we conclude that $G$ is solvable if and only if $H$ is solvable.

  It remains to show that $G$ and $H$ have the same derived length. We argue by induction on the derived length. If the derived length of $G$ is $1$, then $G$ is abelian, and hence $\mathcal{R}(G)$ is a Boolean algebra. Since $\mathcal{R}(G)\cong \mathcal{R}(H)$, the group $H$ is also abelian, so its derived length is $1$.

  Assume now that the statement holds for all finite solvable groups of derived length at most $\ell-1$, and suppose that $G$ has derived length $\ell\geq 2$. Let $N$ be a maximal normal abelian subgroup of $G$. Clearly, the quotient $G/N$ is solvable of derived length at most $\ell-1$. By the construction above we have
  $$ \mathcal{R}(G/N)\setminus \{\emptyset\} \cong \mathcal{R}(H/\psi(N))\setminus \{\emptyset\}. $$
  By the induction hypothesis, the groups $G/N$ and $H/\psi(N)$ have the same derived length. Therefore, the quotient $H/\psi(N)$ has derived length at most $\ell-1$ which implies that the derived length of $H$ is at most $\ell$.

  By symmetry, the same arguments with the roles of $G$ and $H$ interchanged show that the derived length of $G$ is at most the derived length of $H$. Thus $G$ and $H$ have the same derived length, namely $\ell$.

\end{proof}


The following result is proved in \cite[Theorem~3.3]{Kay21}, and the proof of Theorem~\ref{thm:main_derived} given here is developed along the same ideas.

\begin{thm}\label{thm:nilpotence_class}
  If $G$ is a finite nilpotent group of nilpotence class $c$, and if $H$ is a group such that
  $$ \mathcal{R}(G)\cong \mathcal{R}(H), $$
  then $H$ is nilpotent of nilpotence class $c$.
\end{thm}

\begin{proof}
  The proof is analogous to that of Theorem~\ref{thm:main_derived}. The difference is that, instead of considering maximal normal abelian subgroups, one considers the center of the group, which can be determined from the combinatorial structure of the subrack lattice. Indeed, the claim in the proof of Theorem~\ref{thm:main_derived} remains valid when the subgroup $N=Z(G)$. Thus, if $\psi\colon \mathcal{R}(G)\to \mathcal{R}(H)$ is a lattice isomorphism, then it induces an isomorphism
  $$ \mathcal{R}(G/Z(G))\setminus \{\emptyset\}\cong \mathcal{R}(H/Z(H))\setminus \{\emptyset\}. $$
  Repeating this construction along the upper central series, we obtain corresponding central quotients for $G$ and $H$. Therefore the length of the upper central series, that is, the nilpotence class, is determined by the subrack lattice.
\end{proof}

\section{Metacyclic groups}\label{sec:meta}

A group $G$ is called \emph{metacyclic} if it has a cyclic normal subgroup $N$ such that the quotient $G/N$ is also cyclic. The following fact is standard.

\begin{lem}\label{lem:subgroup_quotient}
  Every subgroup and every quotient of a metacyclic group is metacyclic.
\end{lem}

\begin{proof}
  Let $G$ be a group having a cyclic normal subgroup $N$ such that the quotient $G/N$ is also cyclic. If $H$ is a subgroup of $G$, then $H\cap N$ would be a cyclic normal subgroup of $H$. Since $H/(H\cap N)\cong HN/N$ and $HN/N$ is a subgroup of the cyclic group $G/N$, the subgroup $H$ would be metacyclic. Let $K$ be a normal subgroup of $G$. Then $NK/K$ is a cyclic subgroup of $G/K$ with quotient $(G/K)/(NK/K)$ which is also cyclic as $(G/K)/(NK/K)\cong G/NK$ and $G/NK\cong (G/N)/(NK/N)$ is cyclic. Thus, $G/K$ is metacyclic.
\end{proof}

Finite metacyclic groups have been studied extensively in literature. For example, in \cite{Bas69} Basmaji studied their isomorphism problem and in \cite{Hem00} Hempel classified finite metacyclic groups. The subgroup structure of these groups is known to be highly organized. In \cite{Yan20} Yang parametrized all subgroups of a finite metacyclic group by triples of integers and obtained further structural consequences. In all these cited works the following classical theorem of Hölder is fundamental.

\begin{thm}[Hölder, see {\cite[Theorem~21]{Zas99}}]\label{thm:hölder}
  Let $G$ be a finite group of order $mn$, and suppose that $N = \langle a\rangle$ is a cyclic normal subgroup of order $m$ such that $G/N = \langle bN\rangle$ is cyclic of order $n$. Then $G$ admits a presentation
  $$ G = \langle a,b \mid a^m = 1,\ b^n = a^t,\ bab^{-1} = a^r\rangle, $$
  where $m,n,t,r$ satisfy
  \begin{enumerate}[(i)]
  \item $m,n > 0$,
  \item $r^n\equiv 1 \pmod{m}$,
  \item $t(r-1)\equiv 0 \pmod{m}$.
  \end{enumerate}
  Conversely, whenever these numerical conditions hold, the above presentation defines a finite metacyclic group of order $mn$. 
\end{thm}

Let $G$ be a finite metacyclic group. We say a pair of elements $(a,b)$ of $G$ is a \emph{Hölder pair} if $a$ generates a normal subgroup of $G$ and the quotient $G/\langle a\rangle$ is a cyclic group generated by $b\langle a\rangle$. By Theorem~\ref{thm:hölder}, the group $G$ admits a \emph{Hölder presentation}
$$ G = \langle a,b \mid a^m = 1,\ b^n = a^t,\ bab^{-1} = a^r\rangle, $$
satisfying the numerical conditions stated there. In such a case we may write $$ G = \langle a,b\mid m,n,t,r\rangle $$ for brevity. 

For an element $a$ of a group $G$, we denote by $K_G(a)$ the conjugacy class of $a$ in $G$. Recall that the inner automorphism $\phi_a$ is defined as the permutation of $G$ that takes $g$ to $aga^{-1}$ and the set of inner automorphisms of $G$, denoted $\mathrm{Inn}(G)$, form a normal subgroup of the full automorphism group $\mathrm{Aut}(G)$ of $G$. 

\begin{prop}\label{prop:properties}
  Let $G = \langle a,b\mid m',n',t',r'\rangle$ be a finite metacyclic group. The following statements hold.
  \begin{enumerate}[(i)]
  \item Every element of $G$ can be written uniquely in the form $a^ib^j$ (or uniquely in the form $b^ja^i$) for some $0\leq i < m'$ and $0\leq j < n'$.
  \item $C_G(a) = \langle a,b^n\rangle$ and $C_G(b) = \langle a^m,b\rangle$, where $a^m$ and $b^n$ are smallest powers of $a$ and $b$ that are lying in $Z(G)$. Moreover, $Z(G) = \langle a^m,b^n\rangle$.
  \item $K_G(a) = \{b^iab^{-i}\colon i\in \mathbb{Z}\}$, so the conjugacy class of $a$ is exactly the $\langle b\rangle$-orbit of $a$ under conjugation, and therefore $|K_G(a)| = [G:C_G(a)]$ divides $|b|$. Similarly, $K_G(b) = \{a^iba^{-i}\colon i\in \mathbb{Z}\}$ and $[G:C_G(b)]$ divides $|a|$.
  \item $|\phi_a| = m = |K_G(b)| = [G:C_G(b)]$ and $|\phi_b| = n = |K_G(a)| = [G:C_G(a)]$.
  \end{enumerate}  
\end{prop}

\begin{proof}
  \begin{enumerate}[(i)]
  \item This follows from the fact that $\langle a \rangle$ is a normal subgroup of $G$ and the elements of $\langle b\rangle$ form a set of representatives for $G/\langle a\rangle$.
  \item The element $a^ib^j$ centralizes $a$ if and only if $b^j$ centralizes $a$. Observe that if $b^{j_1}$ and $b^{j_2}$ both centralize $a$, then $b^{\mathrm{gcd(j_1,j_2)}}$ centralizes $a$ as well. Therefore, there exists a smallest positive integer $n$ such that $b^n$ centralizes $a$. Similarly, there is a smallest positive integer $m$ such that $a^m$ centralizes $b$. Since $b^n$ commutes both $a$ and $b$, and $a^m$ commutes with both $a$ and $b$, both are central.  Therefore, we see that $C_G(a) = \langle a,b^n\rangle$, $C_G(b) = \langle a^m,b\rangle$, and $Z(G) = \langle a^m,b^n\rangle$.
  \item Observe that $b^ja^ia(b^ja^i)^{-1} = b^jab^{-j}$ since $a^i$ commutes with $a$. Similarly we have $a^ib^jb(a^ib^j)^{-1} = a^iba^{-i}$ as $b^j$ commutes with $b$.
  \item Since $a$ commutes with every power of $a$, for every $i,j$, the $\langle\phi_a\rangle$-orbit of the element $a^ib^j$ has the same length as the $\langle\phi_a\rangle$-orbit of $b^j$. If a power of $a$ centralizes $b$, then it centralizes $b^j$ for every $j$, and therefore the orbit length of $b^j$ divides the orbit length of $b$. Thus every orbit length of $\phi_a$ divides the orbit length of $b$. Similarly, every orbit length of $\phi_b$ divides the orbit length of $a$ and the result follows.
  \end{enumerate}
\end{proof}

The representation of the elements of $G$ in Proposition~\ref{prop:properties}~(i) is the \emph{standard form} for the given Hölder presentation. Since for a group rack the rack operation is the group conjugation operation, it is more convenient for us to represent the elements of $G$ in a \emph{standard form modulo the center} as in Lemma~\ref{lem:rep}~(ii).

\begin{lem}\label{lem:rep}
Let $G = \langle a,b\mid m',n',t',r'\rangle$ with $Z(G) = \langle a^m,b^n\rangle$. Enumerate the elements of $Z(G)$ as $z_k$, $1\leq k \leq m'n'/mn$. Then the following statements hold.
\begin{enumerate}[(i)]
\item If $g=a^ib^j$ and $h=a^ub^v$ are given in the standard form as described in Proposition~\ref{prop:properties}~(i), then
  $$ g\triangleright h = ghg^{-1} = a^wb^v\quad \text{where } w\equiv i + u(r')^j - i(r')^v \pmod{m'}. $$ 
\item Every element of $G$ can be written uniquely in the form $a^ib^jz_k$ for some $0\leq i < m$, $0\leq j < n$, and $1\leq k \leq m'n'/mn$.
\item If $g=a^ib^jz_k$ and $h=a^ub^vz_l$ written in standard form modulo the center, then there exists a unique central element $z_{l'}\in Z(G)$ such that
  $$ g\triangleright h=ghg^{-1} = a^{\bar{w}}b^vz_{l'}\quad \text{where } \bar{w}\equiv i + ur^j - ir^v \pmod{m}. $$
  Equivalently, if $w$ is the unique integer with $0\leq w<m'$ satisfying $$ w\equiv i+u(r')^j-i(r')^v \pmod{m'}, $$ and if $w=\bar{w} + em$ with $0\leq \bar{w} < m$, then $z_{l'}=a^{em}z_l$.
\end{enumerate}
\end{lem}

\begin{proof}
  If $g=a^ib^j$ and $h=a^ub^v$, then using the relation $ba=a^{r'}b$ one obtains $$ ghg^{-1}=a^{\,i+u(r')^j-i(r')^v}b^v, $$ where the exponent is reduced modulo $m'$ and Part~(i) follows.
  
  By Proposition~\ref{prop:properties} an element of $G$ can be expressed uniquely as $a^{i'}b^{j'}$ for some $0\leq i' < m'$ and $0\leq j' < n'$. Then $a^{i'}b^{j'} = a^ib^jz_k$ where $i' = i + mu$, $j' = j + nv$, and $z_k = a^{mu}b^{nv}$. Since $i,j,u$, and $v$ are uniquely determined, Part~(ii) follows.

  Let $g=a^i b^j z_k$ and $h=a^u b^v z_l$. Since $z_k,z_l\in Z(G)$, we have $$ ghg^{-1}=a^ib^j(a^u b^v)b^{-j}a^{-i}z_l. $$ By Part~(i), this is equal to
  $$ a^w b^v z_l\quad \text{where } w\equiv i+u(r')^j-i(r')^v \pmod{m'}. $$
  Write uniquely $w = \bar{w} + em$, $0\leq \bar w < m$. Since $a^{em}\in Z(G)$, we obtain
  $$ a^w b^v z_l = a^{\bar w} b^v (a^{em}z_l). $$
  Thus $ghg^{-1} = a^{\bar{w}}b^vz_{l'}$ for the unique central element $z_{l'}=a^{em}z_l$. This proves Part~(iii).
\end{proof}

\begin{rem}
If $g\in C_G(x)$, then the action of $\phi_g$ on the conjugacy class $K_G(y)$ depends only on the image of $g$ in the cyclic quotient $C_G(x)/Z(G)$. In particular, the size of the $\langle \phi_g\rangle$-orbit of $y$ is equal to the order of $gZ(G)$ in $C_G(x)/Z(G)$.
\end{rem}

The \emph{Euler totient function} $\varphi(n)$ is defined to be the number of units in $(\mathbb{Z}/n\mathbb{Z})^{\times}$. Equivalently, $\varphi(n)$ is the number of integers in $\{1,\dots, n\}$ that are relatively prime to $n$. It is well known that
$$ \sum\limits_{d\mid n}\varphi(d) = n. $$

\begin{prop}\label{prop:conditions}
  Let $G$ be a finite metacyclic group and let $(x,y)$ be a Hölder pair for $G$. Then the following statements hold.
  \begin{enumerate}[(i)]
  \item The centralizer $C_G(x)$ is a maximal normal abelian subgroup of $G$, and
  $$ Z(G)=C_G(x)\cap C_G(y). $$
  \item Let $[C_G(x):Z(G)]=m$. Then for each positive divisor $d$ of $m$, there are exactly
  $$ \varphi(d)\,|Z(G)| $$
  elements $g\in C_G(x)$ such that the image of $g$ in $C_G(x)/Z(G)$ has order $d$.
  \item Let $g\in C_G(x)$. If the image of $g$ in $C_G(x)/Z(G)$ has order $d$, then
  $$ |\langle g,y\rangle_{\mathsf{rk}}\cap K_G(y)|=d. $$
  Consequently, for each positive divisor $d$ of $m$, there are exactly
  $$ \varphi(d)\,|Z(G)| $$
  elements $g\in C_G(x)$ such that
  $$ |\langle g,y\rangle_{\mathsf{rk}}\cap K_G(y)|=d. $$
  \item Let $g,h\in C_G(x)$. If the cyclic subgroup generated by $gZ(G)$ is contained in the cyclic subgroup generated by $hZ(G)$ in $C_G(x)/Z(G)$, then
  $$ \langle g,y\rangle_{\mathsf{rk}}\cap K_G(y)\subseteq \langle h,y\rangle_{\mathsf{rk}}\cap K_G(y). $$
  In particular, if $gZ(G)$ and $hZ(G)$ have the same order in $C_G(x)/Z(G)$, then
  $$ \langle g,y\rangle_{\mathsf{rk}}\cap K_G(y)= \langle h,y\rangle_{\mathsf{rk}}\cap K_G(y). $$
  \end{enumerate}
\end{prop}

\begin{proof}
  Since $(x,y)$ is a Hölder pair for $G$, Proposition~\ref{prop:properties} gives $C_G(x)=\langle x,y^n\rangle$,\; $C_G(y)=\langle x^m,y\rangle$ for suitable positive integers $m,n$, and $Z(G)=C_G(x)\cap C_G(y)$. In particular, $C_G(x)$ is abelian and normal, and it is maximal among normal abelian subgroups, since any subgroup of $G$ properly containing $C_G(x)$ contains an element outside $C_G(x)$ and therefore cannot be abelian. This proves~(i).

  Now $C_G(x)/Z(G)\cong \langle x\rangle/\langle x^m\rangle$ is cyclic of order $m$. Hence, for each positive divisor $d$ of $m$, the quotient $C_G(x)/Z(G)$ contains exactly $\varphi(d)$ elements of order $d$. Since each coset of $Z(G)$ contains exactly $|Z(G)|$ elements, it follows that there are exactly $\varphi(d)|Z(G)|$ elements $g\in C_G(x)$ whose image in $C_G(x)/Z(G)$ has order $d$. This proves~(ii).

  Let $g\in C_G(x)$, and let $d$ be the order of $gZ(G)$ in $C_G(x)/Z(G)$. Since $g\in C_G(x)$, we may write $g=x^iz$ for some integer $i$ and some $z\in Z(G)$. By Lemma~\ref{lem:rep}, conjugation by $g$ acts on elements written in standard form modulo the center by changing the $x$-exponent according to the image of $g$ in $C_G(x)/Z(G)$; the central factor may change, but the action on the conjugacy class $K_G(y)$ is determined entirely by the coset $gZ(G)$. Therefore the orbit of $y$ under $\langle \phi_g\rangle$ has size equal to the order of $gZ(G)$, namely $d$. Since $\langle g,y\rangle_{\mathsf{rk}}\cap K_G(y)$ is precisely this orbit, we obtain $|\langle g,y\rangle_{\mathsf{rk}}\cap K_G(y)|=d$. This proves~(iii).

  For~(iv), suppose that $\langle gZ(G)\rangle\leq \langle hZ(G)\rangle$ in $C_G(x)/Z(G)$. Then there exists an integer $u$ such that $gZ(G)=(hZ(G))^u$. Again by Lemma~\ref{lem:rep}, the action of $\langle \phi_g\rangle$ on $K_G(y)$ is induced by the action of the cyclic subgroup generated by $gZ(G)$ in $C_G(x)/Z(G)$, and similarly for $h$. Hence every element in the $\langle \phi_g\rangle$-orbit of $y$ also lies in the $\langle \phi_h\rangle$-orbit of $y$, that is, $\langle g,y\rangle_{\mathsf{rk}}\cap K_G(y)\subseteq \langle h,y\rangle_{\mathsf{rk}}\cap K_G(y)$. If $gZ(G)$ and $hZ(G)$ have the same order, then they generate the same subgroup of the cyclic group $C_G(x)/Z(G)$, and equality follows.
\end{proof}

\begin{rem}\label{rem:compatible}
  The inclusion relations described in Proposition~\ref{prop:conditions} are compatible with the choice of the second element outside $C_G(x)$. More precisely, let $g,h\in C_G(x)$ and let $w\in G\setminus C_G(x)$. Then the following statements hold.
  \begin{enumerate}[(i)]
  \item If
  $$ \langle g,y\rangle_{\mathsf{rk}}\cap K_G(y)\subseteq \langle h,y\rangle_{\mathsf{rk}}\cap K_G(y), $$
  then
  $$ \langle g,w\rangle_{\mathsf{rk}}\cap K_G(w)\subseteq \langle h,w\rangle_{\mathsf{rk}}\cap K_G(w). $$
  \item If
  $$ \langle g,y\rangle_{\mathsf{rk}}\cap K_G(y) = \langle h,y\rangle_{\mathsf{rk}}\cap K_G(y), $$
  then
  $$ \langle g,w\rangle_{\mathsf{rk}}\cap K_G(w) = \langle h,w\rangle_{\mathsf{rk}}\cap K_G(w). $$
  \end{enumerate}

  Indeed, writing elements of $G$ in standard form modulo the center, if $g=x^iz\in C_G(x)$ and $w=x^uy^vz'\notin C_G(x)$, then the action of $\langle g\rangle$ on the conjugacy class $K_G(w)$ is equivalent to the action of $\langle x^i\rangle$ on the conjugacy class $K_G(y^v)$. Hence the above inclusion and equality relations are determined by the subgroup generated by $gZ(G)$ in the cyclic quotient $C_G(x)/Z(G)$, and therefore the same relations hold with $y$ replaced by $w$.
\end{rem}

\begin{lem}\label{lem:group}
  Let $G$ be a finite group and $(x,y)$ be a pair of elements of $G$. The following conditions are determined by the subrack lattice $\mathcal{R}(G)$.
  \begin{itemize}
  \item[{\rm (H1)}] The centralizer $C_G(x)$ is a maximal normal abelian subgroup of $G$, and $Z(G)=C_G(x)\cap C_G(y)$.
  \item[{\rm (H2)}] Let $[C_G(x):Z(G)]=m$. Then for each positive divisor $d$ of $m$, there are exactly $\varphi(d)|Z(G)|$ elements $g\in C_G(x)$ such that $|\langle g,y\rangle_{\mathsf{rk}}\cap K_G(y)|=d$.
  \item[{\rm (H3)}] Let $g,h\in C_G(x)$. Then $|\langle g,y\rangle_{\mathsf{rk}}\cap K_G(y)|$ divides $|\langle h,y\rangle_{\mathsf{rk}}\cap K_G(y)|$ if and only if $\langle g,y\rangle_{\mathsf{rk}}\cap K_G(y)$ is a subset of $\langle h,y\rangle_{\mathsf{rk}}\cap K_G(y)$.
  \end{itemize}
  Suppose for the pair $(x,y)$ the conditions {\rm (H1)}, {\rm (H2)}, and {\rm (H3)} are satisfied. For a divisor $d$ of $m$, let
  $$ A_d := \{g\in C_G(x)\colon |\langle g,y\rangle_{\mathsf{rk}}\cap K_G(y)|\; \text{ divides }\; d\}. $$
  Then $A_d$ is a normal abelian subgroup of $G$ containing $Z(G)$. Moreover, the quotient $A_d/Z(G)$ is cyclic of order $d$.     
\end{lem}

\begin{rem}\mbox{}
  \begin{enumerate}
  \item If $G$ is a finite metacyclic group and $(x,y)$ is a Hölder pair for $G$, then by Proposition~\ref{prop:conditions} the conditions (H1), (H2), and (H3) are satisfied for $(x,y)$.
  \item Let $G$ be a finite group and $(x,y)$ be a pair of elements of $G$ satisfying (H1), (H2), and (H3). For $g\in C_G(x)$, set $\Omega(g):=\langle g,y\rangle_{\mathsf{rk}}\cap K_G(y)$. Then, by (H3), inclusion among the sets $\Omega(g)$ corresponds exactly to divisibility of their cardinalities. Moreover, by (H2), for each positive divisor $d$ of $m=[C_G(x):Z(G)]$, there exist elements $g\in C_G(x)$ such that $|\Omega(g)|=d$. It follows that the collection
    $$ \{\Omega(g)\colon g\in C_G(x)\}, $$
    partially ordered by inclusion, forms a poset isomorphic to the divisibility lattice $D_m$ via the map $ \Omega(g)\longmapsto |\Omega(g)|$. In turn, for each divisor $d\mid m$, the set
    $$ A_d=\{g\in C_G(x)\colon |\Omega(g)|\mid d\} $$
    defines a subgroup of $C_G(x)$ (by Lemma~\ref{lem:group}), and the family $\{A_d\colon d\mid m\}$ forms a poset under inclusion with minimal element $Z(G)$ and maximal element $C_G(x)$ that is also isomorphic to $D_m$.
  \end{enumerate}
\end{rem}

\begin{proof}
  Since the subrack lattice determines the center, conjugacy classes, centralizers, maximal normal abelian subgroups, and subracks generated by pairs of atoms, the conditions (H1), (H2), and (H3) are encoded in $\mathcal R(G)$. This proves the first assertion of the Lemma.

  For the second assertion, let $(x,y)$ satisfy (H1), (H2), and (H3), and let $d$ be a positive integer dividing $m=[C_G(x):Z(G)]$. Define
  $$ \Omega(g):=\langle g,y\rangle_{\mathsf{rk}}\cap K_G(y). $$
  Since $A_d\subseteq C_G(x)$ and $C_G(x)$ is abelian, $A_d$ is abelian once it is a subgroup. Let $g\in A_d$. Since $g$ and $g^{-1}$ generate the same cyclic subgroup modulo $Z(G)$, the sets $\langle g,y\rangle_{\mathsf{rk}}\cap K_G(y)$ and $\langle g^{-1},y\rangle_{\mathsf{rk}}\cap K_G(y)$ coincide. Hence $g^{-1}\in A_d$. Now let $g,h\in A_d$. Choose an element $c\in C_G(x)$ such that $|\Omega(c)|=d$. By (H3), we have $\Omega(g)\subseteq \Omega(c)$ and $\Omega(h)\subseteq \Omega(c)$. For an element $p\in C_G(x)$, let us denote by $M_p$ the subgroup generated by the $y$-conjugates of $p$. Then $\Omega(p)=\langle p,y\rangle_{\mathsf{rk}}\cap K_G(y)$ is precisely the orbit of $y$ under the action of $M_p$ on $K_G(y)$.

  \medskip
  \noindent
  \emph{Claim. If $p\in C_G(x)$ satisfies $\Omega(p)\subseteq \Omega(c)$, then $M_p$ preserves $\Omega(c)$ setwise.}

  \medskip
  Let $f\in \Omega(c)$. Then $f=qyq^{-1}$ for some $q\in M_c\leq C_G(x)$. Since $C_G(x)$ is abelian, for every generator $y^\ell p y^{-\ell}$ of $M_p$ we have $$ (y^\ell p y^{-\ell})\, f\, (y^\ell p y^{-\ell})^{-1} = q((y^\ell p y^{-\ell})\, y\, (y^\ell p y^{-\ell})^{-1})q^{-1}. $$
  Now
  $$ (y^\ell p y^{-\ell})\, y\, (y^\ell p y^{-\ell})^{-1}\in \Omega(y^\ell p y^{-\ell}) = y^\ell \Omega(p) y^{-\ell}\subseteq y^\ell \Omega(c) y^{-\ell}=\Omega(c), $$
  because $\Omega(c)$ is invariant under conjugation by $y$. Hence every generator of $M_p$ preserves $\Omega(c)$, and therefore $M_p$ preserves $\Omega(c)$. This proves the claim.

  \medskip
  Applying the claim with $p=g$ and $p=h$, we see that both $M_g$ and $M_h$ preserve $\Omega(c)$. On the other hand, for every integer $\ell$ we have
  $$ y^\ell (gh) y^{-\ell}=(y^\ell g y^{-\ell})(y^\ell h y^{-\ell}), $$
  so every $y$-conjugate of $gh$ lies in the subgroup generated by the $y$-conjugates of $g$ and $h$. Therefore $M_{gh}\leq \langle M_g,M_h\rangle$. Hence $M_{gh}$ also preserves $\Omega(c)$, and consequently $\Omega(gh)\subseteq \Omega(c)$. Thus $|\Omega(gh)|$ divides $d$, so $gh\in A_d$. Hence $A_d$ is a subgroup of $C_G(x)$.

  To prove normality, let $t\in G$ and $g\in A_d$. Then $\Omega(tgt^{-1}) = t\Omega(g)t^{-1}$, hence $|\Omega(tgt^{-1})|=|\Omega(g)|$. It follows that $tgt^{-1}\in A_d$, and thus $A_d\trianglelefteq G$.

  Finally, by (H2) and (H3), the sets $A_d/Z(G)$, $d\mid m$, form a lattice under inclusion isomorphic to the divisibility lattice $D_m$. In particular, there is a unique subgroup of each order dividing $m$, hence $A_m/Z(G)$ is cyclic of order $m$. Therefore $A_d/Z(G)$ is cyclic of order $d$.
\end{proof}

By Lemma~\ref{lem:subgroup_quotient}, if $G$ is a finite metacyclic group, then the inner automorphism group $\mathrm{Inn}(G)\cong G/Z(G)$ is also metacyclic. Lemma~\ref{lem:group} suggests that a relation of the first type in a Hölder presentation for $\mathrm{Inn}(G)$ can be determined from the combinatorial properties of $\mathcal R(G)$. To recover the order of the second generator, one may perform an analogous analysis inside $C_G(y)$ by studying the action of elements of $C_G(y)$ on the conjugacy class $K_G(x)$.

\begin{lem}\label{lem:group_y}
  Let $G$ be a finite group and $(x,y)$ be a pair of elements of $G$. The following conditions are determined by the subrack lattice $\mathcal{R}(G)$.
  \begin{itemize}
  \item[{\rm (H1)}] The centralizer $C_G(x)$ is a maximal normal abelian subgroup of $G$, and $Z(G)=C_G(x)\cap C_G(y)$.
  \item[{\rm (H2$'$)}] Let $[C_G(y):Z(G)]=n$. Then for each positive divisor $d$ of $n$, there are exactly $\varphi(d)|Z(G)|$ elements $w\in C_G(y)$ such that $|\langle x,w\rangle_{\mathsf{rk}}\cap K_G(x)|=d$.
  \item[{\rm (H3$'$)}] Let $w,w'\in C_G(y)$. Then $|\langle x,w\rangle_{\mathsf{rk}}\cap K_G(x)|$ divides $|\langle x,w'\rangle_{\mathsf{rk}}\cap K_G(x)|$ if and only if $\langle x,w\rangle_{\mathsf{rk}}\cap K_G(x)$ is a subset of $\langle x,w'\rangle_{\mathsf{rk}}\cap K_G(x)$.
  \end{itemize}
  Suppose for the pair $(x,y)$ the conditions {\rm (H1)}, {\rm (H2$'$)}, and {\rm (H3$'$)} are satisfied. For a divisor $d$ of $n$, let
  $$ B_d := \{w\in C_G(y)\colon |\langle x,w\rangle_{\mathsf{rk}}\cap K_G(x)|\; \text{ divides }\; d\}. $$
  Then $B_d$ is an abelian subgroup of $G$ containing $Z(G)$. Moreover, the quotient $B_d/Z(G)$ is cyclic of order $d$.     
\end{lem}

\begin{rem}
  If $G$ is a finite metacyclic group and $(x,y)$ is a Hölder pair for $G$, then the conditions (H1), (H2$'$), and (H3$'$) are satisfied for $(x,y)$. Indeed, the proof is obtained from Proposition~\ref{prop:conditions} by interchanging the roles of $x$ and $y$, while using the standard form modulo the center. More precisely, since
  $$ C_G(y)/Z(G)\cong \langle y\rangle/\langle y^n\rangle $$
  is cyclic of order $n=[C_G(y):Z(G)]$, the action of an element $w\in C_G(y)$ on the conjugacy class $K_G(x)$ depends only on the image of $w$ in $C_G(y)/Z(G)$. Hence, for each divisor $d$ of $n$, there are exactly $\varphi(d)|Z(G)|$ elements $w\in C_G(y)$ such that
  $$ |\langle x,w\rangle_{\mathsf{rk}}\cap K_G(x)|=d, $$
  and inclusion among these subsets is equivalent to divisibility of their cardinalities.
\end{rem}

\begin{proof}
  The first assertion is proved exactly as in Lemma~\ref{lem:group}: the center, conjugacy classes, centralizers, and subracks generated by pairs of atoms are determined by $\mathcal{R}(G)$, so the conditions (H1), (H2$'$), and (H3$'$) are encoded in the subrack lattice.

  For the second assertion, define
  $$ \Xi(w):=\langle x,w\rangle_{\mathsf{rk}}\cap K_G(x),\qquad w\in C_G(y). $$
  The proof is parallel to that of Lemma~\ref{lem:group}. In particular, if $d\mid n=[C_G(y):Z(G)]$ and $c\in C_G(y)$ satisfies $|\Xi(c)|=d$, then for any $q\in C_G(y)$ with $\Xi(q)\subseteq \Xi(c)$, the subgroup generated by the $x$-conjugates of $q$ preserves $\Xi(c)$ setwise. Since $C_G(y)$ is abelian, it follows exactly as before that
  $$ B_d=\{w\in C_G(y)\colon |\Xi(w)|\text{ divides }d\} $$
  is a subgroup of $C_G(y)$ containing $Z(G)$. Hence $B_d$ is abelian. Finally, by (H2$'$) and (H3$'$), the family $\{B_d\colon d\mid n\}$ is ordered by inclusion as the divisibility lattice $D_n$. Therefore, exactly as in Lemma~\ref{lem:group}, the quotient $B_d/Z(G)$ is cyclic of order $d$.
\end{proof}

The next proposition shows that a relation of the second type in a Hölder presentation for $\mathrm{Inn}(G)$ can also be determined from the combinatorial properties of $\mathcal R(G)$.

\begin{prop}\label{prop:second_relation}
  Let $G$ be a finite group and let $x,y\in G$ satisfy $Z(G)=C_G(x)\cap C_G(y)$. Then $\phi_x^i=\phi_y^j$ if and only if the order of $\phi_x$ divides $i$ and the order of $\phi_y$ divides $j$.
\end{prop}

\begin{proof}
  If the orders of $\phi_x$ and $\phi_y$ divide $i$ and $j$, respectively, then $\phi_x^i=1=\phi_y^j$. Conversely, suppose that $\phi_x^i=\phi_y^j$. Since $\phi_x^i=\phi_{x^i}$ and $\phi_y^j=\phi_{y^j}$, we have $\phi_{x^i}=\phi_{y^j}$. Evaluating both sides at $y$, we obtain $x^i y x^{-i}=y^j y y^{-j}=y$, so $x^i\in C_G(y)$. Similarly, evaluating both sides at $x$, we get $x^i x x^{-i}=y^j x y^{-j}$, hence $y^j\in C_G(x)$. Now $x^i\in C_G(y)\cap C_G(x)=Z(G)$, and similarly $y^j\in C_G(x)\cap C_G(y)=Z(G)$. Hence $\phi_x^i=\phi_{x^i}=1$ and $\phi_y^j=\phi_{y^j}=1$. Therefore the order of $\phi_x$ divides $i$ and the order of $\phi_y$ divides $j$.
\end{proof}

We have now recovered from $\mathcal{R}(G)$ the data needed to reconstruct the two cyclic factors and the conjugation relation in $\mathrm{Inn}(G)$. We now combine the previous lemmas and propositions to prove the second main result of this paper.

\begin{proof}[Proof of Theorem~\ref{thm:main_metacyclic}]
  Let $G = \langle a,b\mid m',n',t',r'\rangle$ and set $m:=|K_G(b)|=[G:C_G(b)]$ and $n:=|K_G(a)|=[G:C_G(a)]$ so that $Z(G) = \langle a^m,b^n\rangle$. By Proposition~\ref{prop:act}, we have $\mathrm{Inn}(G)\cong G/Z(G)$, and hence, by Lemma~\ref{lem:subgroup_quotient}, the group $\mathrm{Inn}(G)$ is metacyclic. Moreover, Proposition~\ref{prop:act} shows that $\mathrm{Inn}(G)$ is generated by $\phi_a$ and $\phi_b$, while Proposition~\ref{prop:properties}(iv) gives $|\phi_a|=m$ and $|\phi_b|=n$. Since $\langle a\rangle\trianglelefteq G$, the subgroup $\langle \phi_a\rangle$ is normal in $\mathrm{Inn}(G)$, and there exists an integer $r$ such that $\phi_b\phi_a\phi_b^{-1}=\phi_a^r$. Because $|\phi_b|=n$, we obtain $\phi_a=\phi_b^n\phi_a\phi_b^{-n}=\phi_a^{\,r^n}$, and therefore $r^n\equiv 1 \pmod{m}$. Furthermore, by Proposition~\ref{prop:second_relation}, if $\phi_b^i=\phi_a^j$, then $n\mid i$ and $m\mid j$. Hence the quotient $\mathrm{Inn}(G)/\langle \phi_a\rangle$ is cyclic of order $n$ generated by $\phi_b\langle \phi_a\rangle$. It follows that $(\phi_a,\phi_b)$ is a Hölder pair for $\mathrm{Inn}(G)$ and that
  $$ \mathrm{Inn}(G) = \langle \phi_a,\phi_b \mid \phi_a^m=1,\ \phi_b^n=1,\ \phi_b\phi_a\phi_b^{-1}=\phi_a^r\rangle, $$
  where $r^n\equiv 1\pmod{m}$.
  
  Let $\psi\colon \mathcal{R}(G)\to \mathcal{R}(H)$ be a lattice isomorphism and let the images of the atoms $\{a\}$ and $\{b\}$ under $\psi$ be $\{\alpha\}$ and $\{\beta\}$, respectively. Since $(a,b)$ is a Hölder pair for $G$, Proposition~\ref{prop:conditions} shows that the pair $(a,b)$ satisfies (H1), (H2), and (H3). Although these conditions are stated in group-theoretic terms, they are determined by the combinatorial structure of the subrack lattice. Hence, because $\psi\colon \mathcal{R}(G)\longrightarrow \mathcal{R}(H)$ is a lattice isomorphism sending the atoms $\{a\}$ and $\{b\}$ to the atoms $\{\alpha\}$ and $\{\beta\}$, respectively, the pair $(\alpha,\beta)$ satisfies the corresponding conditions in $H$. By Lemma~\ref{lem:group}, the quotient $C_H(\alpha)/Z(H)$ is cyclic of order $m=[C_G(a):Z(G)]=|K_G(b)|$. Therefore $|\phi_\alpha|=m$, and hence the relation  $\phi_\alpha^m=1$ holds in $\mathrm{Inn}(H)$.

  By the remark after Lemma~\ref{lem:group_y}, the pair $(a,b)$ also satisfies (H2$'$) and (H3$'$). Therefore the pair $(\alpha,\beta)$ satisfies (H2$'$) and (H3$'$) in $H$. By Lemma~\ref{lem:group_y}, the quotient $C_H(\beta)/Z(H)$ is cyclic of order $n=[C_G(b):Z(G)]=|K_G(a)|$. Therefore $|\phi_\beta|=n$, and hence the relation $\phi_\beta^n=1$ holds in $\mathrm{Inn}(H)$.

  Next, since $C_H(\alpha)$ is a normal subgroup of $H$, we have $\beta\alpha\beta^{-1}\in C_H(\alpha)$. Hence there exists an element $\gamma\in C_H(\alpha)$ such that $\phi_\beta\phi_\alpha\phi_\beta^{-1}=\phi_\gamma$. Since $C_H(\alpha)/Z(H)$ is cyclic of order $m$, there exists an integer $\bar{r}$ such that $\gamma Z(H)=\alpha^{\bar r}Z(H)$. Therefore $\phi_\gamma=\phi_\alpha^{\bar r}$, and so $\phi_\beta\phi_\alpha\phi_\beta^{-1}=\phi_\alpha^{\bar r}$. This gives the third type relation in $\mathrm{Inn}(H)$.

  Now, because $|\phi_\beta|=n$, we have $\phi_\beta^n=1$, and therefore $\phi_\beta^n\phi_\alpha\phi_\beta^{-n}=\phi_\alpha$. On the other hand, iterating the relation $\phi_\beta\phi_\alpha\phi_\beta^{-1}=\phi_\alpha^{\bar r}$ yields $\phi_\beta^n\phi_\alpha\phi_\beta^{-n}=\phi_\alpha^{\bar r^{\,n}}$. Thus $\phi_\alpha^{\bar r^{\,n}}=\phi_\alpha$. Since $|\phi_\alpha|=m$, it follows that $\bar r^{\,n}\equiv 1 \pmod{m}$. Hence the numerical conditions of Hölder's theorem are satisfied for the presentation $\langle\phi_\alpha,\phi_\beta\mid m,n,0,\bar{r}\rangle$, and therefore $\langle \phi_\alpha,\phi_\beta\rangle$ is a split metacyclic group of order $mn$.

  On the other hand, because $\psi$ is a lattice isomorphism, the groups $G$ and $H$ have the same number of elements. Moreover, the center is determined by the subrack lattice, so $|Z(G)|=|Z(H)|$. Therefore
  $$ |\mathrm{Inn}(H)|=[H:Z(H)]=[G:Z(G)]=|\mathrm{Inn}(G)|. $$
  Since $\mathrm{Inn}(G)$ is a split metacyclic group of order $mn$, we obtain $|\mathrm{Inn}(H)|=mn$. But $\langle \phi_\alpha,\phi_\beta\rangle$ is a subgroup of $\mathrm{Inn}(H)$ of order $mn$. Consequently, $\mathrm{Inn}(H)=\langle \phi_\alpha,\phi_\beta\rangle$. In particular, $\mathrm{Inn}(H)$ is metacyclic, and it admits a presentation
  $$ \mathrm{Inn}(H) = \langle \phi_\alpha,\phi_\beta \mid \phi_\alpha^m=1,\ \phi_\beta^n=1,\ \phi_\beta\phi_\alpha\phi_\beta^{-1}=\phi_\alpha^{\bar r}\rangle, $$
  where $\bar r^{\,n}\equiv 1 \pmod m$. This completes the proof.

\end{proof}

\section{Isomorphism problem for subrack lattices}\label{sec:iso}

In this section we prove Proposition~\ref{prop:ex}. We begin with a general observation showing that, for finite racks, the full subrack lattice is determined by the subracks generated by pairs of elements.

\begin{prop}\label{prop:pairs}
  Let $X$ and $Y$ be two finite racks, and let $f\colon X\to Y$ be a bijection satisfying
  $$ f(\langle x,x'\rangle_{\mathsf{rk}}) = \langle f(x),f(x')\rangle_{\mathsf{rk}} $$
  for all $x,x'\in X$. Then $f$ induces a lattice isomorphism
  $$ f\colon\mathcal{R}(X)\to \mathcal{R}(Y). $$
\end{prop}

\begin{proof}
  Let $A\subseteq X$. We define a sequence of subsets of $X$ by
  $$ A_0:=A,\qquad A_{n+1}:=A_n\cup\bigcup_{x,x'\in A_n}\langle x,x'\rangle_{\mathsf{rk}} \quad (n\geq 0). $$
  Since $X$ is finite, this ascending sequence stabilizes at some subset $A_\infty$. We claim that $A_\infty=\langle A\rangle_{\mathsf{rk}}$. Indeed, each $A_n$ is contained in every subrack containing $A$, because if a subrack contains $A_n$, then it also contains $\langle x,x'\rangle_{\mathsf{rk}}$ for all $x,x'\in A_n$, and hence contains $A_{n+1}$. Therefore $A_\infty$ is contained in every subrack containing $A$. On the other hand, if $x,x'\in A_\infty$, then $x,x'\in A_n$ for some $n$, so $\langle x,x'\rangle_{\mathsf{rk}}\subseteq A_{n+1}\subseteq A_\infty$. Thus $A_\infty$ is itself a subrack containing $A$. Hence $A_\infty=\langle A\rangle_{\mathsf{rk}}$.

  Now let $B:=f(A)\subseteq Y$, and define $B_0,B_1,\dots$ similarly by
  $$ B_0:=B,\qquad B_{n+1}:=B_n\cup\bigcup_{y,y'\in B_n}\langle y,y'\rangle_{\mathsf{rk}}. $$
  We prove by induction on $n$ that $B_n=f(A_n)$ for all $n\geq 0$. This is clear for $n=0$. Assume that $B_n=f(A_n)$. Then
  \begin{align*}
    B_{n+1}
      &= B_n\cup\bigcup_{y,y'\in B_n}\langle y,y'\rangle_{\mathsf{rk}} \\
      &= f(A_n)\cup\bigcup_{x,x'\in A_n}\langle f(x),f(x')\rangle_{\mathsf{rk}} \\
      &= f(A_n)\cup\bigcup_{x,x'\in A_n} f(\langle x,x'\rangle_{\mathsf{rk}}) \\
      &= f\!\left(A_n\cup\bigcup_{x,x'\in A_n}\langle x,x'\rangle_{\mathsf{rk}}\right) \\
      &= f(A_{n+1}),
  \end{align*}
  proving the induction step. Passing to the stable values of the two sequences, we obtain
  $$ f(\langle A\rangle_{\mathsf{rk}})=f(A_\infty)=B_\infty=\langle f(A)\rangle_{\mathsf{rk}}. $$
  In particular, if $S$ is a subrack of $X$, then $f(S)=\langle f(S)\rangle_{\mathsf{rk}}$ is a subrack of $Y$. Thus $f$ induces a map
  $$ \mathcal{R}(X)\to\mathcal{R}(Y),\qquad S\mapsto f(S). $$
  Applying the same argument to $f^{-1}$, we see that this map is bijective. Since $f$ preserves inclusion of subsets, it is an isomorphism of posets, hence a lattice isomorphism.
\end{proof}

Recall that the groups $G_0$ and $H_0$ are defined by
$$ G_0 = \langle a,b\mid 91,3,0,9\rangle\qquad \text{and}\qquad H_0 = \langle \alpha,\beta\mid 91,3,0,16\rangle. $$

\begin{lem}\label{lem:g0h0}
  The groups $G_0$ and $H_0$ are not isomorphic.
\end{lem}

\begin{proof}
  Suppose $f\colon G_0\to H_0$ is a group isomorphism. Since both $G_0$ and $H_0$ are semidirect products $C_{91}\rtimes C_3$, the subgroups $\langle a \rangle$ and $\langle \alpha \rangle$ are the unique normal subgroups of order $91$. Hence $f(\langle a \rangle)=\langle \alpha \rangle$ implying $f(a) = \alpha^u$ for some integer $u$ that is relatively prime to $91$. Since $b\notin \langle a \rangle$, its image does not lie in $\langle \alpha\rangle$, so we may write $f(b) = \alpha^i\beta^j$ for some integer $i$ and some $j\in \{1,2\}$. Now $f(bab^{-1}) = f(a^9) = \alpha^{9u}$. On the other hand,
  $$ f(b)f(a)f(b)^{-1} = (\alpha^i\beta^j)\alpha^u(\alpha^i\beta^j)^{-1} = \alpha^{16^ju}. $$
  Therefore $\alpha^{9u}=\alpha^{16^ju}$. Since $\mathrm{gcd}(u,91)=1$, it follows that $9\equiv 16^j\pmod{91}$. However, this is not the case as $16^2\equiv 74\not\equiv 9\pmod{91}$. This is a contradiction. Hence $G_0$ and $H_0$ are not isomorphic.
\end{proof}

In proving Proposition~\ref{prop:ex} we shall show that following construction yields a lattice isomorphism from $\mathcal{R}(G_0)$ to $\mathcal{R}(H_0)$.

\begin{construction}
  We define a bijection $\psi_0\colon G_0\to H_0$ as follows. First, consider the normal cyclic subgroups $\langle a\rangle\leq G_0$ and $\langle \alpha\rangle\leq H_0$. For each divisor $d\in\{1,7,13,91\}$, let
  $$ A_d:=\{a^i\in \langle a\rangle : |a^i|=d\}, \qquad A_d':=\{\alpha^i\in \langle \alpha\rangle : |\alpha^i|=d\}. $$
  Equivalently,
  \begin{align*}
    A_1 &= \{a^0\},\\
    A_7 &= \{a^i : \gcd(i,91)=13\},\\
    A_{13} &= \{a^i : \gcd(i,91)=7\},\\
    A_{91} &= \{a^i : \gcd(i,91)=1\},
  \end{align*}
  and similarly for $A_1',A_7',A_{13}',A_{91}'$.

  For each $d\in\{1,7,13,91\}$, the conjugation action of $b$ on $A_d$ decomposes $A_d$ into conjugacy classes of the form
  $$ \{a^i,a^{9i},a^{81i}\}, $$
  while the conjugation action of $\beta$ on $A_d'$ decomposes $A_d'$ into conjugacy classes of the form
  $$ \{\alpha^i,\alpha^{16i},\alpha^{74i}\}. $$
  We order the conjugacy classes in $A_d$ and $A_d'$ by increasing least exponent, and within each class we order the elements by increasing exponent. We then define $\psi_0$ on $\langle a\rangle$ by sending the $k$th element of the $j$th conjugacy class in $A_d$ to the $k$th element of the $j$th conjugacy class in $A_d'$.

  Next, for the coset $\langle a\rangle b$, we define
  $$ \psi_0(a^i b)=\alpha^i\beta \qquad (0\leq i<91). $$

  Finally, for the coset $\langle a\rangle b^2$, we define $\psi_0$ so that the centralizers of the elements $a^ib$ are preserved. Since
  $$ (a^i b)^2=a^{10i}b^2 \qquad\text{and}\qquad (\alpha^i\beta)^2=\alpha^{17i}\beta^2, $$
  the centralizers of $a^i b$ and $\alpha^i\beta$ contain the elements $a^{10i}b^2$ and $\alpha^{17i}\beta^2$, respectively. Thus we require
  $$ \psi_0(a^{10i}b^2)=\alpha^{17i}\beta^2. $$
  Equivalently, if
  $$ \psi_0(a^k b^2)=\alpha^{ck}\beta^2, $$
  then $c$ must satisfy
  $$ 10c\equiv 17\pmod{91}. $$
  Since $c\equiv 29\pmod{91}$ is the unique solution, we define
  $$ \psi_0(a^i b^2)=\alpha^{29i}\beta^2 \qquad (0\leq i<91). $$
\end{construction}

The following lemma is useful to determine the subracks of a group that are generated by a pair of its elements.

\begin{lem}\label{lem:pair}
  Let $G$ be a group and let $x,y\in G$. Set $L:=\langle x,y\rangle$. Then
  $$ \langle x,y\rangle_{\mathsf{rk}}=K_L(x)\cup K_L(y). $$
\end{lem}

\begin{proof}
  Since every subrack of $G$ containing $x$ and $y$ is closed under conjugation by elements generated from $x$ and $y$, it contains every conjugate of $x$ and every conjugate of $y$ by elements of $L$. Hence
  $$ K_L(x)\cup K_L(y)\subseteq \langle x,y\rangle_{\mathsf{rk}}. $$
  Conversely, every element obtained from $x$ and $y$ by repeated rack operations is obtained by conjugation by elements of $L$, and therefore is a conjugate in $L$ of either $x$ or $y$. Thus
  $$ \langle x,y\rangle_{\mathsf{rk}}\subseteq K_L(x)\cup K_L(y). $$
  Therefore
  $$ \langle x,y\rangle_{\mathsf{rk}}=K_L(x)\cup K_L(y). $$
\end{proof}

Now we are ready to prove our final main result. The following proof gives a direct verification that $\psi_0$ preserves the subracks generated by pairs of elements. A GAP \cite{GAP4} implementation of this verification is included in the appendix.

\begin{proof}[Proof of Proposition~\ref{prop:ex}]
  First we show that $G_0$ and $H_0$ are not isomorphic as racks. Since $Z(G_0)=Z(H_0)=1$, any isomorphism between the group racks $G_0$ and $H_0$ would induce an isomorphism of groups $G_0\to H_0$. By Lemma~\ref{lem:g0h0}, the groups $G_0$ and $H_0$ are not isomorphic. Therefore $G_0\not\cong_{\mathsf{rk}} H_0$.

  Next we show that the map $\psi_0\colon G_0\to H_0$ induces a lattice isomorphism. We verify the equality
  $$ \psi_0(\langle x,y\rangle_{\mathsf{rk}}) = \langle \psi_0(x),\psi_0(y)\rangle_{\mathsf{rk}} $$
  for every pair $x,y\in G_0$. Since every element of $G_0$ lies in exactly one of the three cosets $N=\langle a\rangle$, $Nb$, and $Nb^2$, every pair falls into exactly one of the six unordered coset-types
  $$ (N,N),\ (N,Nb),\ (N,Nb^2),\ (Nb,Nb),\ (Nb,Nb^2),\ (Nb^2,Nb^2). $$
  In each nonabelian case the argument follows the same pattern: we set $L:=\langle x,y\rangle$, determine $L\cap N=\langle a^d\rangle$ for a suitable divisor $d$ of $91$, observe that $(a^d,x)$ or $(a^d,y)$ is a Hölder pair for $L$, describe the relevant conjugacy classes as $\langle a^d\rangle$-orbits, and then apply Lemma~\ref{lem:pair}. We then repeat the same analysis for $L':=\langle \psi_0(x),\psi_0(y)\rangle$, using that $\psi_0$ preserves the order classes $A_1,A_7,A_{13},A_{91}$ inside $N$ and maps each conjugacy class in $N$ onto the corresponding conjugacy class in $\langle \alpha\rangle$.

  \medskip
  \noindent
  \emph{Case~1.\; $\psi_0(\langle a^{i_1},a^{i_2}\rangle_{\mathsf{rk}}) = \langle \psi_0(a^{i_1}),\psi_0(a^{i_2})\rangle_{\mathsf{rk}}$}

  \medskip
  Let $x=a^{i_1}$ and $y=a^{i_2}$. Since $N$ and $\psi_0(N)=\langle\alpha\rangle$ are abelian, both $\langle x,y\rangle_{\mathsf{rk}}$ and $\langle \psi_0(x),\psi_0(y)\rangle_{\mathsf{rk}}$ are just the subracks generated by two commuting elements. Hence the desired equality follows.

  \medskip
  \noindent
  \emph{Case~2.\; $\psi_0(\langle a^{i_1},a^{i_2}b\rangle_{\mathsf{rk}})=\langle \psi_0(a^{i_1}),\psi_0(a^{i_2}b)\rangle_{\mathsf{rk}}$}
  
  \medskip
  If $i_1=0$, then $a^{i_1}=1$ is central, so $\langle a^{i_1},a^{i_2}b\rangle_{\mathsf{rk}}=\{1,a^{i_2}b\}$. Since $\psi_0(1)=1$ and $\psi_0(a^{i_2}b)=\alpha^{i_2}\beta$, we also have $\langle \psi_0(a^{i_1}),\psi_0(a^{i_2}b)\rangle_{\mathsf{rk}}=\{1,\alpha^{i_2}\beta\}$, and the desired equality follows.
  
  Assume now that $i_1\neq 0$. Set
  $$ x:=a^{i_1},\qquad y:=a^{i_2}b,\qquad L:=\langle x,y\rangle. $$
  Let $d:=\gcd(91,i_1)$ and $m:=91/d$. Then $|x|=m$, and
  $$ yxy^{-1} = (a^{i_2}b)a^{i_1}(a^{i_2}b)^{-1} = ba^{i_1}b^{-1} = a^{9i_1} = x^9. $$
  Also, $(a^{i_2}b)^3=1$. Hence $(x,y)$ is a Hölder pair for $L$ with $L=\langle x,y\mid m,3,0,9\rangle$. Since $(x,y)$ is a Hölder pair for $L$, the conjugacy class of $x$ in $L$ is the
  $\langle y\rangle$-orbit of $x$, namely
  $$ K_L(x)=\{x,x^9,x^{81}\}=\{a^{i_1},a^{9i_1},a^{81i_1}\}. $$
  Similarly, the conjugacy class of $y$ in $L$ is the $\langle x\rangle$-orbit of $y$. Since $x^n y x^{-n}=a^{i_2-8n i_1}b$ and $-8$ is a unit modulo $91$, we obtain
  $$ K_L(y)=\{a^j b : j\in i_2+\langle i_1\rangle\}, $$
  where $\langle i_1\rangle$ denotes the additive subgroup of $\mathbb{Z}/91\mathbb{Z}$ generated by $i_1$. Hence, by Lemma~\ref{lem:pair},
  $$ \langle a^{i_1},a^{i_2}b\rangle_{\mathsf{rk}} = K_L(x)\cup K_L(y) = K_{G_0}(a^{i_1})\cup\{a^j b : j\in i_2+\langle i_1\rangle\}. $$
  
  Now set
  $$ x':=\psi_0(a^{i_1}),\qquad y':=\psi_0(a^{i_2}b)=\alpha^{i_2}\beta,\qquad L':=\langle x',y'\rangle. $$
  Since $\psi_0$ preserves the order of elements of $\langle a\rangle$, we have $|x'|=|x|=m$. Moreover,
  $$ y'x'(y')^{-1}=\beta x'\beta^{-1}=(x')^{16}, $$
  and $(y')^3=1$. Thus $(x',y')$ is a Hölder pair for $L'$ with $L'=\langle x',y'\mid m,3,0,16\rangle$. Since $\psi_0$ preserves the order of $a^{i_1}$ and hence the additive subgroup $\langle i_1\rangle\leq \mathbb{Z}/91\mathbb{Z}$, we have
  $$ K_{L'}(x')=K_{H_0}(x')\quad \text{and}\quad K_{L'}(y')=\{\alpha^j\beta : j\in i_2+\langle i_1\rangle\}. $$
  Hence, by Lemma~\ref{lem:pair},
  $$ \langle x',y'\rangle_{\mathsf{rk}} = K_{L'}(x')\cup K_{L'}(y') = K_{H_0}(\psi_0(a^{i_1}))\cup \{\alpha^j\beta : j\in i_2+\langle i_1\rangle\}. $$

  \medskip
  \noindent
  \emph{Case~3.\; $\psi_0(\langle a^{i_1},a^{i_2}b^2\rangle_{\mathsf{rk}})=\langle \psi_0(a^{i_1}),\psi_0(a^{i_2}b^2)\rangle_{\mathsf{rk}}$}
  
  \medskip
  If $i_1=0$, then $a^{i_1}=1$ is central, so $\langle a^{i_1},a^{i_2}b^2\rangle_{\mathsf{rk}}=\{1,a^{i_2}b^2\}$. Since $\psi_0(1)=1$ and $\psi_0(a^{i_2}b^2)=\alpha^{29i_2}\beta^2$, we also have $\langle \psi_0(a^{i_1}),\psi_0(a^{i_2}b^2)\rangle_{\mathsf{rk}}=\{1,\alpha^{29i_2}\beta^2\}$, and the desired equality follows.
  
  Assume now that $i_1\neq 0$. Set
  $$ x:=a^{i_1},\qquad y:=a^{i_2}b^2,\qquad L:=\langle x,y\rangle. $$
  Let $d:=\gcd(91,i_1)$ and $m:=91/d$. Then $|x|=m$, and
  $$ yxy^{-1} = (a^{i_2}b^2)a^{i_1}(a^{i_2}b^2)^{-1} = b^2a^{i_1}b^{-2} = a^{81i_1} = x^{81}. $$
  Also, $(a^{i_2}b^2)^3=1$. Hence $(x,y)$ is a Hölder pair for $L$ with $L=\langle x,y\mid m,3,0,81\rangle$ and the conjugacy class of $x$ in $L$ is the $\langle y\rangle$-orbit of $x$, namely
  $$ K_L(x)=\{x,x^{81},x^{81^2}\}=\{a^{i_1},a^{81i_1},a^{9i_1}\}=K_{G_0}(a^{i_1}). $$
  Similarly, the conjugacy class of $y$ in $L$ is the $\langle x\rangle$-orbit of $y$. Since
  $$ x^n y x^{-n} = a^{ni_1}(a^{i_2}b^2)a^{-ni_1} = a^{\,i_2-80ni_1}b^2, $$
  and $-80$ is a unit modulo $91$, we obtain
  $$ K_L(y)=\{a^j b^2 : j\in i_2+\langle i_1\rangle\}, $$
  where $\langle i_1\rangle$ denotes the additive subgroup of $\mathbb{Z}/91\mathbb{Z}$ generated by $i_1$. Hence, by Lemma~\ref{lem:pair},
  $$ \langle a^{i_1},a^{i_2}b^2\rangle_{\mathsf{rk}} = K_L(x)\cup K_L(y) = K_{G_0}(a^{i_1})\cup \{a^j b^2 : j\in i_2+\langle i_1\rangle\}. $$
  
  Now set
  $$ x':=\psi_0(a^{i_1}),\qquad y':=\psi_0(a^{i_2}b^2)=\alpha^{29i_2}\beta^2,\qquad L':=\langle x',y'\rangle. $$
  Since $\psi_0$ preserves the order of elements of $\langle a\rangle$, we have $|x'|=|x|=m$. Also,
  $$ y'x'(y')^{-1} = (\alpha^{29i_2}\beta^2)x'(\alpha^{29i_2}\beta^2)^{-1} = \beta^2x'\beta^{-2} = (x')^{74}, $$
  because if $x'=\alpha^u$, then $\beta^2\alpha^u\beta^{-2}=\alpha^{16^2u}=\alpha^{74u}$. Moreover, $(y')^3=1$. Hence $(x',y')$ is a Hölder pair for $L'$ with $L' = \langle x',y'\mid m,3,0,74\rangle$. Therefore the conjugacy class of $x'$ in $L'$ is the $\langle y'\rangle$-orbit of $x'$,
  that is,
  $$ K_{L'}(x')=K_{H_0}(x'). $$
  Similarly, the conjugacy class of $y'$ in $L'$ is the $\langle x'\rangle$-orbit of $y'$. If
  $x'=\alpha^u$, then
  $$ (x')^n y' (x')^{-n} = \alpha^{nu}(\alpha^{29i_2}\beta^2)\alpha^{-nu} = \alpha^{\,29i_2-73nu}\beta^2. $$
  Since $\psi_0$ preserves the gcd-layer of $i_1$, the additive subgroups generated by $u$, $i_1$, and $29i_1$ in $\mathbb{Z}/91\mathbb{Z}$ all coincide. As $-73$ is a unit modulo
  $91$, it follows that
  $$ K_{L'}(y')=\{\alpha^k\beta^2 : k\in 29i_2+\langle 29i_1\rangle\}. $$
  On the other hand, $\psi_0(a^j b^2)=\alpha^{29j}\beta^2$, so
  \begin{align*}
    \psi_0(\{a^j b^2:j\in i_2+\langle i_1\rangle\})
    &= \{\alpha^{29j}\beta^2:j\in i_2+\langle i_1\rangle\} \\
    &= \{\alpha^k\beta^2:k\in 29i_2+\langle 29i_1\rangle\}.
  \end{align*}
  Thus
  $$ \psi_0(K_L(y))=K_{L'}(y') \qquad\text{and}\qquad \psi_0(K_L(x))=K_{L'}(x'). $$
  Hence, by Lemma~\ref{lem:pair},
  $$ \psi_0(\langle a^{i_1},a^{i_2}b^2\rangle_{\mathsf{rk}}) = K_{L'}(x')\cup K_{L'}(y') = \langle \psi_0(a^{i_1}),\psi_0(a^{i_2}b^2)\rangle_{\mathsf{rk}}. $$

  \medskip
  \noindent
  \emph{Case~4.\; $\psi_0(\langle a^{i_1}b,a^{i_2}b\rangle_{\mathsf{rk}}) = \langle \psi_0(a^{i_1}b),\psi_0(a^{i_2}b)\rangle_{\mathsf{rk}}$}
  
  \medskip
  Set
  $$ x:=a^{i_1}b,\qquad y:=a^{i_2}b,\qquad L:=\langle x,y\rangle. $$
  Let $d:=\gcd(91,i_1,i_2)$. Then $L\cap \langle a\rangle = \langle a^d\rangle$ and $(a^d,x)$ is a Hölder pair for $L$ with $L/\langle a^d\rangle$ is cyclic of order $3$. Therefore the conjugacy class of $x$ in $L$ is the $\langle a^d\rangle$-orbit of $x$. Since
  $$ a^{nd}xa^{-nd}=a^{nd}(a^{i_1}b)a^{-nd}=a^{\,i_1-8nd}b, $$
  and $-8$ is a unit modulo $91$, we obtain
  $$ K_L(x)=\{a^{i_1+t}b : t\in \langle d\rangle\}, $$
  where $\langle d\rangle$ denotes the additive subgroup of $\mathbb{Z}/91\mathbb{Z}$ generated by $d$. Since $x$ and $y$ are conjugate in $L$, by Lemma~\ref{lem:pair} we have
  $$ \langle a^{i_1}b,a^{i_2}b\rangle_{\mathsf{rk}} = K_L(x)\cup K_L(y) = \{a^{i_1+t}b : t\in \langle d\rangle\}. $$
  
  By construction,
  $$ \psi_0(a^jb)=\alpha^j\beta \qquad (0\leq j<91), $$
  so
  $$ \psi_0(\{a^{i_1+t}b : t\in \langle d\rangle\}) = \{\alpha^{i_1+t}\beta : t\in \langle d\rangle\}. $$
  Now set
  $$ x':=\alpha^{i_1}\beta,\qquad y':=\alpha^{i_2}\beta,\qquad L':=\langle x',y'\rangle. $$
  As above $L'\cap \langle \alpha\rangle=\langle \alpha^d\rangle$ and $(\alpha^d,x')$ is a Hölder pair for $L'$ with $L'/\langle\alpha^d \rangle$ is cyclic of order $3$. Therefore the conjugacy class of $x'$ in $L'$ is the $\langle \alpha^d\rangle$-orbit of $x'$. Since
  $$ \alpha^{nd}x'\alpha^{-nd} = \alpha^{nd}(\alpha^{i_1}\beta)\alpha^{-nd} = \alpha^{\,i_1-15nd}\beta, $$
  and $-15$ is a unit modulo $91$, we obtain
  $$ K_{L'}(x')=\{\alpha^{i_1+t}\beta : t\in \langle d\rangle\}. $$
  Since $x'$ and $y'$ are conjugate in $L'$, by Lemma~\ref{lem:pair} we have 
  $$ \langle x',y'\rangle_{\mathsf{rk}} = K_{L'}(x')\cup K_{L'}(y') = \{\alpha^{i_1+t}\beta : t\in \langle d\rangle\}. $$
  Therefore
  $$ \psi_0(\langle a^{i_1}b,a^{i_2}b\rangle_{\mathsf{rk}}) = \langle \alpha^{i_1}\beta,\alpha^{i_2}\beta\rangle_{\mathsf{rk}} = \langle \psi_0(a^{i_1}b),\psi_0(a^{i_2}b)\rangle_{\mathsf{rk}}. $$

  \medskip
  \noindent
  \emph{Case~5.\; $\psi_0(\langle a^{i_1}b,a^{i_2}b^2\rangle_{\mathsf{rk}}) = \langle \psi_0(a^{i_1}b),\psi_0(a^{i_2}b^2)\rangle_{\mathsf{rk}}$}
  
  \medskip
  Set
  $$ x:=a^{i_1}b,\qquad y:=a^{i_2}b^2,\qquad L:=\langle x,y\rangle. $$
  Let $d:=\gcd(91,i_1,i_2)$. Then $L\cap \langle a\rangle = \langle a^d\rangle$ and $(a^d,x)$ is a Hölder pair for $L$ with $L/\langle a^d\rangle$ is cyclic of order $3$. Therefore the conjugacy class of $x$ in $L$ is the $\langle a^d\rangle$-orbit of $x$. Since
  $$ a^{nd}xa^{-nd}=a^{nd}(a^{i_1}b)a^{-nd}=a^{\,i_1-8nd}b, $$
  and $-8$ is a unit modulo $91$, we obtain
  $$ K_L(x)=\{a^{i_1+t}b : t\in \langle d\rangle\}, $$
  where $\langle d\rangle$ denotes the additive subgroup of $\mathbb{Z}/91\mathbb{Z}$
  generated by $d$. Similarly, the conjugacy class of $y$ in $L$ is the $\langle a^d\rangle$-orbit of $y$. Since
  $$ a^{nd}ya^{-nd}=a^{nd}(a^{i_2}b^2)a^{-nd}=a^{\,i_2-80nd}b^2, $$
  and $-80$ is a unit modulo $91$, we obtain
  $$ K_L(y)=\{a^{i_2+t}b^2 : t\in \langle d\rangle\}. $$
  Hence, by Lemma~\ref{lem:pair},
  $$ \langle a^{i_1}b,a^{i_2}b^2\rangle_{\mathsf{rk}} = K_L(x)\cup K_L(y) = \{a^{i_1+t}b : t\in \langle d\rangle\} \,\cup\, \{a^{i_2+t}b^2 : t\in \langle d\rangle\}. $$
  
  Now set
  $$ x':=\psi_0(x)=\alpha^{i_1}\beta,\qquad y':=\psi_0(y)=\alpha^{29i_2}\beta^2,\qquad L':=\langle x',y'\rangle. $$
  As above $L'\cap \langle \alpha\rangle=\langle \alpha^d\rangle$ and $(\alpha^d,x')$ is a Hölder pair for $L'$ with $L'/\langle\alpha^d \rangle$ is cyclic of order $3$. Since the conjugacy class of $x'$ in $L'$ is the $\langle \alpha^d\rangle$-orbit of $x'$, we obtain
  $$
  K_{L'}(x')=\{\alpha^{i_1+t}\beta : t\in \langle d\rangle\}.
  $$
  Similarly, the conjugacy class of $y'$ in $L'$ is the $\langle \alpha^d\rangle$-orbit of $y'$ and we have
  $$ K_{L'}(y')=\{\alpha^{29i_2+t}\beta^2 : t\in \langle d\rangle\}. $$
  
  By construction,
  $$ \psi_0(a^{i_1+t}b)=\alpha^{i_1+t}\beta \qquad\text{and}\qquad \psi_0(a^{i_2+t}b^2)=\alpha^{29(i_2+t)}\beta^2. $$
  Since multiplication by $29$ is a unit modulo $91$, the sets
  $$ \{\alpha^{29(i_2+t)}\beta^2:t\in \langle d\rangle\} \qquad\text{and}\qquad \{\alpha^{29i_2+t}\beta^2:t\in \langle d\rangle\} $$
  coincide. Hence
  $$ \psi_0(K_L(x))=K_{L'}(x') \qquad\text{and}\qquad \psi_0(K_L(y))=K_{L'}(y'). $$
  Therefore, again by Lemma~\ref{lem:pair},
  $$ \psi_0(\langle a^{i_1}b,a^{i_2}b^2\rangle_{\mathsf{rk}}) = \langle \alpha^{i_1}\beta,\alpha^{29i_2}\beta^2\rangle_{\mathsf{rk}} = \langle \psi_0(a^{i_1}b),\psi_0(a^{i_2}b^2)\rangle_{\mathsf{rk}}. $$

  \medskip
  \noindent
  \emph{Case~6.\; $\psi_0(\langle a^{i_1}b^2,a^{i_2}b^2\rangle_{\mathsf{rk}})
    =\langle \psi_0(a^{i_1}b^2),\psi_0(a^{i_2}b^2)\rangle_{\mathsf{rk}}$}
  
  \medskip
  Set
  $$ x:=a^{i_1}b^2,\qquad y:=a^{i_2}b^2,\qquad L:=\langle x,y\rangle. $$
  Let $d:=\gcd(91,i_1,i_2)$. Similar to previous cases $(a^d,x)$ is a Hölder pair for $L$ and the conjugacy class of $x$ in $L$ is the $\langle a^d\rangle$-orbit of $x$. Since
  $$ a^{nd}xa^{-nd}=a^{nd}(a^{i_1}b^2)a^{-nd}=a^{\,i_1-80nd}b^2, $$
  and $-80$ is a unit modulo $91$, we have
  $$ K_L(x)=\{a^{i_1+t}b^2:t\in \langle d\rangle\}. $$
  Since $x$ and $y$ are conjugate in $L$, by Lemma~\ref{lem:pair} we have
  $$ \langle a^{i_1}b^2,a^{i_2}b^2\rangle_{\mathsf{rk}} = K_L(x)\cup K_L(y) = \{a^{i_1+t}b^2 : t\in \langle d\rangle\}. $$
  
  As was observed before
  $$ \psi_0(\{a^{i_1+t}b^2:t\in \langle d\rangle\}) = \{\alpha^{29(i_1+t)}\beta^2:t\in \langle d\rangle\} = \{\alpha^{29i_1+t}\beta^2:t\in \langle d\rangle\}. $$
  Now set
  $$ x':=\alpha^{29i_1}\beta^2,\qquad y':=\alpha^{29i_2}\beta^2,\qquad L':=\langle x',y'\rangle. $$
  Similar to Case~4, $(\alpha^d,x')$ is a Hölder pair for $L'$ and the conjugacy class of $x'$ in $L'$ is the $\langle \alpha^d\rangle$-orbit of $x'$. Since
  $$ \alpha^{nd}x'\alpha^{-nd} = \alpha^{nd}(\alpha^{29i_1}\beta^2)\alpha^{-nd} = \alpha^{\,29i_1-73nd}\beta^2, $$
  and $-73$ is a unit modulo $91$, we obtain
  $$ K_{L'}(x')=\{\alpha^{29i_1+t}\beta^2:t\in \langle d\rangle\}. $$
  Since $x'$ and $y'$ are conjugate in $L'$, by Lemma~\ref{lem:pair} we have 
  $$ \langle x',y'\rangle_{\mathsf{rk}} = K_{L'}(x')\cup K_{L'}(y') = \{\alpha^{29i_1+t}\beta^2:t\in \langle d\rangle\}. $$
  Therefore
  $$ \psi_0(\langle a^{i_1}b^2,a^{i_2}b^2\rangle_{\mathsf{rk}}) = \langle \alpha^{29i_1}\beta^2,\alpha^{29i_2}\beta^2\rangle_{\mathsf{rk}} = \langle \psi_0(a^{i_1}b^2),\psi_0(a^{i_2}b^2)\rangle_{\mathsf{rk}}. $$

  \medskip
  The six cases above exhaust all possibilities for pairs of elements of $G_0$. Hence $\psi_0$ satisfies the hypothesis of Proposition~\ref{prop:pairs}. Therefore $\psi_0$ induces a lattice isomorphism from $\mathcal{R}(G_0)$ to $\mathcal{R}(H_0)$. Combining this with the first part of the proof, we obtain Proposition~\ref{prop:ex}.

\end{proof}

\paragraph{Author's note.}
The author used a large language model for assistance with exposition and language during the preparation of this manuscript. The author verified all mathematical content and takes full responsibility for the results and presentation.


\section*{Appendix}

The following GAP code uses $G$, $H$ in place of $G_0$, $H_0$.
\begin{lstlisting}[
  language=GAP,
  basicstyle=\ttfamily\footnotesize,
  breaklines=true,
  columns=fullflexible,
  keepspaces=true,
  aboveskip=3pt,
  belowskip=3pt
]
############################################################
# Split metacyclic groups
#   G = <a,b | a^91=1, b^3=1, b*a*b^-1 = a^9>
#   H = <alpha,beta | alpha^91=1, beta^3=1,
#                      beta*alpha*beta^-1 = alpha^16>
############################################################

F := FreeGroup("a","b");;
fa := F.1;;
fb := F.2;;

Gfp := F / [ fa^91, fb^3, fb*fa*fb^-1*fa^-9 ];;

qa := Gfp.1;;
qb := Gfp.2;;

isoG := IsomorphismPcGroup(Gfp);;
G := Image(isoG);;
aG := Image(isoG, qa);;
bG := Image(isoG, qb);;

F2 := FreeGroup("alpha","beta");;
falpha := F2.1;;
fbeta  := F2.2;;

Hfp := F2 / [ falpha^91, fbeta^3, fbeta*falpha*fbeta^-1*falpha^-16 ];;

qalpha := Hfp.1;;
qbeta  := Hfp.2;;

isoH := IsomorphismPcGroup(Hfp);;
H := Image(isoH);;
alphaH := Image(isoH, qalpha);;
betaH  := Image(isoH, qbeta);;

Print("Size(G) = ", Size(G), "\n");
Print("Size(H) = ", Size(H), "\n");

############################################################
# Normal-form constructors: a^i b^j and alpha^i beta^j
############################################################

Mod91 := function(i)
    return i mod 91;
end;;

Mod3 := function(j)
    return j mod 3;
end;;

EltG := function(i,j)
    return aG^Mod91(i) * bG^Mod3(j);
end;;

EltH := function(i,j)
    return alphaH^Mod91(i) * betaH^Mod3(j);
end;;

############################################################
# Enumerate all elements in normal form
############################################################

pairsG := [];; eltsG := [];;

for j in [0..2] do
    for i in [0..90] do
        Add(pairsG, [i,j]);
        Add(eltsG, EltG(i,j));
    od;
od;

pairsH := [];; eltsH := [];;

for j in [0..2] do
    for i in [0..90] do
        Add(pairsH, [i,j]);
        Add(eltsH, EltH(i,j));
    od;
od;

NF_G := function(g)
    local pos;
    pos := Position(eltsG, g);
    if pos = fail then
        Error("Element not found in G.");
    fi;
    return pairsG[pos];
end;;

NF_H := function(h)
    local pos;
    pos := Position(eltsH, h);
    if pos = fail then
        Error("Element not found in H.");
    fi;
    return pairsH[pos];
end;;

############################################################
# Orbits modulo 91
############################################################

OrbitMod := function(r, u, m)
    local x, orb;
    x := u mod m;
    orb := [];
    repeat
        AddSet(orb, x);
        x := (r * x) mod m;
    until x = u mod m;
    Sort(orb);
    return orb;
end;;

############################################################
# Orbits inside a fixed gcd-layer
############################################################

AllOrbitsModInLayer := function(r, m, d)
    local seen, orbs, u, orb;
    seen := [];
    orbs := [];

    for u in [0..m-1] do
        if Gcd(u,m) = d and not u in seen then
            orb := OrbitMod(r, u, m);;
            orb := Filtered(orb, x -> Gcd(x,m) = d);;
            Sort(orb);
            Add(orbs, orb);
            UniteSet(seen, orb);
        fi;
    od;

    Sort(orbs);
    return orbs;
end;;

############################################################
# Define expmap on Z/91Z by matching 9-orbits with 16-orbits
# inside each gcd-layer d in {0,1,7,13}
############################################################

expmap := [];;

for d in [0,1,7,13] do
    if d = 0 then
        expmap[1] := 0;
    else
        orbs9d  := AllOrbitsModInLayer(9, 91, d);;
        orbs16d := AllOrbitsModInLayer(16, 91, d);;

        Print("layer d = ", d,
              " : #9-orbits = ", Length(orbs9d),
              " , #16-orbits = ", Length(orbs16d), "\n");

        if Length(orbs9d) <> Length(orbs16d) then
            Error("Orbit counts do not match in layer d = ", d);
        fi;

        for t in [1..Length(orbs9d)] do
            O9  := orbs9d[t];
            O16 := orbs16d[t];

            if Length(O9) <> Length(O16) then
                Error("Orbit lengths do not match in layer d = ", d,
                      ", orbit number ", t);
            fi;

            for k in [1..Length(O9)] do
                expmap[O9[k] + 1] := O16[k];
            od;
        od;
    fi;
od;

Print("Size(Set(expmap)) = ", Size(Set(expmap)), "\n");

bad := Filtered([0..90], u -> Gcd(u,91) <> Gcd(expmap[u+1],91));;
Print("Number of bad u = ", Length(bad), "\n");
if Length(bad) > 0 then
    Print(bad, "\n");
fi;

############################################################
# Define the bijection f : G -> H
#
# On <a>:
#   a^u |-> alpha^(expmap(u))
#
# On <a>b:
#   a^i b |-> alpha^i beta
#
# On <a>b^2:
#   a^k b^2 |-> alpha^(29 k) beta^2
############################################################

fPair := function(p)
    local i, j;
    i := p[1];
    j := p[2];

    if j = 0 then
        return [ expmap[i+1], 0 ];
    elif j = 1 then
        return [ i, 1 ];
    elif j = 2 then
        return [ (29*i) mod 91, 2 ];
    else
        Error("j must be 0, 1, or 2.");
    fi;
end;;

fElt := function(g)
    local p, q;
    p := NF_G(g);
    q := fPair(p);
    return EltH(q[1], q[2]);
end;;

############################################################
# Sanity checks for the bijection
############################################################

img0 := List([0..90], i -> fElt(EltG(i,0)));;
img1 := List([0..90], i -> fElt(EltG(i,1)));;
img2 := List([0..90], i -> fElt(EltG(i,2)));;

Print("Size(Set(img0)) = ", Size(Set(img0)), "\n");
Print("Size(Set(img1)) = ", Size(Set(img1)), "\n");
Print("Size(Set(img2)) = ", Size(Set(img2)), "\n");

Print("Size(Set(Concatenation(img0,img1,img2))) = ",
      Size(Set(Concatenation(img0,img1,img2))), "\n");

############################################################
# Rack operation and subrack generation
############################################################

RackOp := function(x,y)
    return x * y * x^-1;
end;;

SubrackGenerated := function(lst)
    local S, changed, x, y, z;

    S := Set(lst);
    changed := true;

    while changed do
        changed := false;

        for x in S do
            for y in S do
                z := RackOp(x,y);
                if not z in S then
                    AddSet(S,z);
                    changed := true;
                fi;
            od;
        od;
    od;

    return S;
end;;

ImageUnderf := function(S)
    return Set(List(S, fElt));
end;;

############################################################
# Test whether f preserves all pair-generated subracks
############################################################

PairTest := function()
    local x, y, SX, SY, count, total;

    count := 0;
    total := Size(G)^2;

    for x in eltsG do
        for y in eltsG do
            SX := SubrackGenerated([x,y]);
            SY := SubrackGenerated([fElt(x), fElt(y)]);

            if ImageUnderf(SX) <> SY then
                Print("\nFailure found.\n");
                Print("x in G normal form = ", NF_G(x), "\n");
                Print("y in G normal form = ", NF_G(y), "\n");
                Print("Size(<x,y>) = ", Size(SX), "\n");
                Print("Size(<f(x),f(y)>) = ", Size(SY), "\n");
                return fail;
            fi;

            count := count + 1;
            if count mod 5000 = 0 then
                Print("Checked ", count, " / ", total, " pairs\n");
            fi;
        od;
    od;

    Print("\nAll ", total, " ordered pairs passed.\n");
    return true;
end;;

############################################################
# Final test
############################################################

PairTest();
\end{lstlisting}

\end{document}